\documentclass[11pt]{amsart}
\usepackage{amsmath,amsthm,amsfonts,amssymb,
graphicx,amscd,float,hyperref,enumitem,setspace, marginnote}
\usepackage[margin=1.45in]{geometry}
\usepackage{comment}
\usepackage{graphicx,xypic}

\newtheorem{thm}{Theorem}[section]
\newtheorem{theorem}{Theorem}[section]

\newtheorem{lem}[thm]{Lemma}
\newtheorem{lemma}[thm]{Lemma}

\newtheorem{cor}[thm]{Corollary}
\newtheorem{corollary}[thm]{Corollary}

\newtheorem{proposition}[thm]{Proposition}

\theoremstyle{definition}
\newtheorem{rk}[thm]{Remark}
\newtheorem{conv}[thm]{Convention}
\newtheorem{df}[thm]{Definition}

  \newcommand{\R}{\mathbb{R}}

  \newcommand{\Z}{\mathbb{Z}}

\newcommand{\vphi}{\varphi}

\newcommand{\wh}{\widehat}

\newcommand{\out}{\textup{Out}(F_r)}

\newcommand{\norm}[1]{|#1|}
\DeclareMathOperator{\Out}{Out}

\newcommand{\os}{\rm{CV}}
\newcommand{\cv}{\os}

\newcommand{\FF}{\mathcal{FF}}
\newcommand{\val}{{\rm{val}}}
\newcommand{\vol}{\mathrm{vol}}

\newcommand{\ds}{d_{\mathrm{sym}}}

\newcommand{\dL}{d_{\rm{CV}}}

\newcommand{\ind}{\mbox{ind}}

\newcommand{\supp}{\mathrm{Supp}}
\newcommand{\sm}{\langle \supp(\mu)\rangle_+}

\makeatletter
\@namedef{subjclassname@2020}{%
	\textup{2020} Mathematics Subject Classification}
\makeatother

\begin{document}

\title[Random trees are trivalent]{Random trees in the boundary of Outer space}

\author{Ilya Kapovich, Joseph Maher, Catherine Pfaff, and Samuel J. Taylor}

\address{Department of Mathematics and Statistics, Hunter College of CUNY\newline
  \indent 695 Park Ave, New York, NY 10065
  \newline \indent  {\url{http://math.hunter.cuny.edu/ilyakapo/}}, }
  \email{\tt ik535@hunter.cuny.edu}

\address{ CUNY College of Staten Island and CUNY Graduate Center \newline
  \indent
2800 Victory Boulevard, Staten Island, NY 10314
  \newline
  \indent  {\url{http://www.math.csi.cuny.edu/~maher/}}, }
  \email{\tt joseph.maher@csi.cuny.edu}

\address{Department of Math \& Stats, Queen's University \newline
  \indent
Jeffery Hall, 48 University Ave.,
Kingston, ON Canada, K7L 3N6
  \newline
  \indent  {\url{https://mast.queensu.ca/~cpfaff/}}, }
  \email{\tt c.pfaff@queensu.ca}

\address{Department of Mathematics, Temple University \newline
  \indent
1805 Broad St, Philadelphia, PA 19122
  \newline
  \indent  {\url{https://math.temple.edu/~samuel.taylor/}}, }
  \email{\tt samuel.taylor@temple.edu}


\begin{abstract}
We prove that for the harmonic measure associated to a random
walk on $\out$ satisfying some mild conditions, a  typical tree
in the boundary of Outer space is trivalent and
nongeometric.
This result answers a question of Mladen Bestvina.
\end{abstract}


\maketitle

\section{Introduction} \label{sec:introduction}
As a means to study the outer automorphism group $\out$,  Culler and Vogtmann \cite{cv86} introduced Outer space $\cv_r$ as the deformation space of marked metric $F_r$-graphs.
Outer space is naturally equipped with a boundary $\partial \cv_r$ whose points are represented by actions of $F_r$ on the class of `very small' $\mathbb{R}$-trees  \cite{cl95, bf94}. Since its introduction, $\partial \cv_r$ has attracted much of its own attention and plays a role similar to that of Thurston's boundary of Teichm\"uller space.

Since a point of $\partial \cv_r$ is the homothety class $[T]$ of an $\mathbb{R}$-tree $T$, one can study its basic properties as such. For example, each $p \in T$ separates $T$, and the number of its complementary components is the \emph{valency} of $p$. We call $T$ \emph{trivalent} if each of its branch-points (i.e. points of valency at least $3$) is $3$-valent. Similarly, one can also consider the manner in which $T$ arises as an $F_r$-tree; $T$ is called \emph{geometric} if it is dual to a measured foliation on a $2$-complex whose fundamental group is $F_r$. As a point of reference, \emph{all} of the $\mathbb{R}$-trees that arise in Thurston's boundary of the Teichm\"uller space are geometric since they are dual to singular measured foliations on the underlying surface. Moreover, in that setting, the valencies of the branch-points correspond to the degrees of the singularities on the surface.

In this paper we develop a complete understanding of these two properties for a ``random'' tree in $\partial \cv_r$.
As a significant point of contrast to the surface case, we find that such a random tree of $\partial \cv_r$ is \emph{not} geometric.

For this, let $(w_n)_{n\ge 1}$ be the random walk on $\out$ determined by a nonelementary measure $\mu$ on $\out$. By combining work of Horbez \cite{horbez2016poisson}  and Namazi--Pettet--Reynolds \cite{namazi2014ergodic},
we recall that the random walk
induces a naturally associated \emph{hitting} or \emph{exit} measure $\nu$ on $\partial \cv_r$ and that $\nu$ is the unique $\mu$-stationary probability measure on $\partial \cv_r$.  Moreover, $\nu$ gives full measure to the subspace of trees in $\partial \cv_r$ which are free, arational, and uniquely ergodic.
We refer the reader to Section \ref{sec:background} for the relevant background. Our main theorem is the following:

\begin{theorem} \label{th:main}
Let $r\ge 3$ and let $\mu$ be a nonelementary probability measure on $\out$ with finite support such that
the semigroup generated by the support of $\mu$ contains
$\varphi^{-1} $ for some
principal fully irreducible $\varphi \in \out$.

Then for $\nu$- almost every $[T]\in \partial \cv_r$, the tree $T$ is trivalent and nongeometric.
\end{theorem}

This answers a question of Mladen Bestvina, who asked us whether almost every tree in $\partial \cv_r$ is trivalent.

\vspace{2mm}

An important component of our argument for Theorem \ref{th:main} is the existence of a \emph{principal} outer automorphism in the semigroup generated by the support of $\mu$.
Such outer automorphisms were originally introduced in \cite{stablestrata} and are discussed further in Section \ref{sec:fellow_folding}.
Let us remark here that principal outer automorphisms are analogous to pseudo-Anosov mapping classes whose Teichm\"uller axes live in the top dimensional stratum over Teichm\"uller space.

As a simple example, we note that the hypotheses of Theorem \ref{th:main} are satisfied when the support of $\mu$ is a finite symmetric generating set of $\out$ -- see Corollary~\ref{cor:main} below.

\subsection*{Connections to previous work}
In our previous work \cite{KMPT}, we proved that with probability approaching $1$ as $n \to \infty$, the random outer automorphism $w_n$ is fully irreducible and its attracting/repelling trees $T^{w_n}_{\pm}$ are trivalent and nongeometric.
However, since such trees form a countable, and hence $\nu$-measure zero, subset of $\partial \cv_r$, this provides no information about a $\nu$-typical tree in $\partial \cv_r$. Indeed, the machinery previously employed, that of ideal Whitehead graphs associated to fully irreducible outer automorphisms, is no longer available in the general setting studied in this paper. Instead, we rely on new results that connect the structure of folding paths to properties of their limiting trees in order to study branching and index properties of the latter.

Our main theorem (Theorem \ref{th:main}) in some sense parallels, and is inspired by, the main theorem of \cite{gm17} in the mapping class group setting. There, Gadre--Maher show that with respect to the hitting measure, a typical lamination in Thurston's boundary of Teichm\"uller space has complementary regions that are triangles and once-punctured disks.

However, our setting differs from theirs in a few key ways. First, their arguments ultimately rely on the openness of the top dimensional stratum in the unit cotangent bundle of Teichm\"uller space. Of course there is no similar structure for $\cv_r$ and
so entirely different techniques must be developed. For this, we introduce the concepts of \emph{eventually legalizing} folding rays (Section \ref{sec:condition*}) and \emph{principal recurrence} (Section \ref{sec:random_rays}) which we hope will additionally be useful in future work.
Second, as previously mentioned, in the mapping class group setting every limit point of the random walk is geometric (essentially by definition), and so the fact that a typical tree in $\partial \cv_r$ is nongeometric is a truly novel feature of the $\out$-setting. Our argument for this
uses the index theory of   Gaboriau and Levitt \cite{gl95}. Informally, this states that being nongeometric is equivalent to the failure of a `Poincar\'e--Hopf index formula' for branch-points of the tree. Using our specialized folding rays, we show that such a formula typically fails.

\subsection*{Outline of paper}
Section \ref{sec:background} provides background on some geometric tools used to study $\out$ and concludes by discussing a few properties of the hitting measure on the boundary of Outer space associated to a random walk on $\out$. In Section \ref{sec:fellow_folding}, we discuss the needed properties of principal outer automorphisms. These are fully irreducible outer automorphisms whose axes in Outer space have particularly rigid and saturated structure. The main result there (Proposition \ref{prop:same_fold}) says that an arbitrary folding path which closely fellow travels such an axis inherits much of the same structure.

Section \ref{sec:condition*} presents our main (nonrandom) criteria (Theorem \ref{th:trivalent}) ensuring that a folding ray determines a limiting tree that is trivalent and nongeometric. We call such folding paths \emph{eventually legalizing}. Informally, these are folding rays for which every path is, after flowing forward and pulling tight, eventually legal, i.e. no longer folded.  If the  `eventually legalizing' condition on the folding ray holds, it allows one to recover the precise structure of the branch-points of the limiting tree $T$ from the graphs along the ray, without losing any directions at the branch-points.
A similar issue arose in a recent paper ~\cite{BHW16}, where the authors introduced a ``carrying index"  of $T$ which sufficed for their purposes but might not detect some directions at branch-points of $T$.

To establish the eventually legalizing property for a \emph{random} folding ray, we introduce the notion of \emph{principal recurrence} in Section \ref{sec:random_rays}. A folding ray is principally recurrent if it fellow travels a translate of a principal axis on arbitrarily long subsegments. The main result (Proposition \ref{prop:random_recurrent}) of Section \ref{sec:random_rays} says that random folding rays are principally recurrent.

Finally, in Section \ref{sec:principal_recurrence} we show that a principally recurrent folding path is eventually legalizing (Proposition \ref{recurrent_star}). The proof of this fact uses results established in Section \ref{sec:fellow_folding} and is another instance of a folding path inheriting the structure of a principal axis that it fellow travels. In Section~\ref{sec:main} we combine the above results to complete the proof of Theorem \ref{th:main}.

\subsection*{Acknowledgments}
The first named author was partially supported by NSF grants DMS-1710868 and DMS-1905641. The second named author was
 supported by Simons Foundation and PSC-CUNY. The third named author acknowledges support from a Queen's University Research Initiation Grant.
The last named author thanks Spencer Dowdall for enlightening conversations and is partially supported by NSF grant DMS-1744551.
 All authors thank M. Bestvina for asking the central question addressed in this paper and acknowledge support from U.S. National
    Science Foundation grants DMS 1107452, 1107263, 1107367 ``RNMS:
    GEometric structures And Representation varieties'' (the GEAR
    Network).
    
Finally, we thank the referee for several helpful suggestions and corrections.

\section{Background}\label{sec:background}

We record here some preliminaries used throughout the paper. Most of this appears in the literature, with exceptions including Proposition \ref{prop:folds_to}, which builds folding paths to trees in $\partial \cv$, and Corollary \ref{cor:random_free}, which establishes that a random tree in $\partial \cv$ is free.

\subsection{Outer space}\label{ss:outerspace}

We denote by $\widehat \cv$ the unprojectivized Outer space for the free group $F_r$ (where $r\ge 2$), and we denote by $\cv = \cv_r$ the corresponding projectivized Outer space.
A point in $\widehat\cv$ is represented (up to some natural equivalence) by a marked metric graph structure on a finite connected graph $G$ where each vertex of $G$ has degree $\ge 3$, the \emph{metric} assigns each edge of $G$ a strictly positive length, and the \emph{marking} identifies $\pi_1(G)$ with $F_r$. We can also think of this point of $\widehat\cv$ as the minimal free discrete isometric action of $F_r$ on the $\R$-tree $T=\widetilde G$ with the lifted metric. We denote by $\vol(G)=\vol(T)$ the sum of the lengths of the edges of $G$.  The space $\cv\subseteq \widehat \cv$ consists of points $G\in \widehat \cv$ with $\vol(G)=1$.

There is a natural closure $\overline{\widehat\cv}$ of $\widehat \cv$ with respect to the length function topology, and $\overline{\widehat\cv}$ is known to consist of precisely the \emph{very small} nontrivial minimal isometric actions by $F_r$ on $\R$-trees. The projectivization of $\overline{\widehat\cv}$ with respect to the natural multiplication action of $\R_{>0}$ is denoted $\overline\cv$; it is known that $\overline\cv$ is compact. For every $T\in\widehat\cv$ the projective class $[T]$ is canonically identified with $T/\vol(T)\in \cv$, and thus we can think of $\cv$ as the projectivization of $\widehat\cv$, and so as a subset of $\overline\cv$.  We denote $\partial\cv=\overline\cv-\cv$.
For additional background on Outer space, its topology, and its boundary see \cite{cv86, cl95, bf94, Pa89}.

For $G_1,G_2\in\widehat\cv$, we denote by $\Lambda(G_1,G_2)$ the infimum of the Lipschitz constants of the continuous maps $f \colon G_1\to G_2$ preserving the marking, i.e. ``change of marking" maps. It is known that for $G_1,G_2\in\cv$ we have $\Lambda(G_1,G_2)\ge 1$, and that $G_1=G_2$ in $\cv$ if and only if $\Lambda(G_1,G_2)=1$.  For $G_1,G_2\in\cv$ we denote $\dL(G_1,G_2)=\log\Lambda(G_1,G_2)$ and refer to $\dL$ as the \emph{asymmetric Lipschitz metric} on $\cv$. For more on this metric, see \cite{FrancavigliaMartino,a08,bf11}. As is common, we let $\ds$ denote the \emph{symmetric Lipschitz metric}: $\ds(G_1,G_2) = \dL(G_1,G_2) + \dL(G_2,G_1)$.

For an interval $J\subseteq \R$, a map $\gamma \colon J\to \cv$ is called a \emph{geodesic} in $\cv$ if \linebreak $\dL(\gamma(t),\gamma(t')) =t'-t$ for all $t,t'\in J$ with $t\le t'$.
A \emph{geodesic ray} in $\cv$ is a geodesic $\gamma \colon  [0,\infty)\to \cv$. We emphasize that the term geodesic always refers to the asymmetric Lipschitz metric.

\subsection{Laminations and arational trees}\label{back:arat}

We refer the reader to~\cite{chl08I,chl08II,ReynoldsArational,br,bf11} for detailed background on algebraic laminations on $F_r$, arational trees, and the free factor complex. We only recall a few basic facts here. For a free group $F_r$ (with $r\ge 2$) let $\partial F_r$ be its Gromov boundary and let $\partial^2 F_r=\{(z_1,z_2)\in \partial F_r\times \partial F_r| z_1\ne z_2\}$. The set $\partial^2 F_r$ is equipped with the subspace topology from $ \partial F_r\times \partial F_r$ and with the diagonal translation action of $F_r$. An \emph{algebraic lamination} on $F_r$ is a subset $L\subseteq \partial^2 F_r$ which is closed, $F_r$-invariant, and flip-invariant (for the ``flip'' map $\partial^2 F_r\to \partial^2 F_r$ defined by $(z_1,z_2)\mapsto (z_2,z_1)$). For an algebraic lamination $L$ on $F_r$ a pair $(z_1,z_2)\in L$ is called a \emph{leaf} of $L$.  For a lamination $L$ on $F_r$, a leaf $(z_1,z_2)\in L$, and a nontrivial finitely generated subgroup $H\le F_r$ we say that $(z_1,z_2)$ is \emph{carried} by $H$ if both $z_1$ and $z_2$ are contained in $\partial H$. Here we have used the facts that $H$ is itself free and that the inclusion $H \to F_r$ induces an embedding $\partial H \to \partial F_r$.

For any tree $T\in \overline{\widehat\cv}$ there is an associated \emph{dual lamination} or \emph{zero lamination}  $L(T)\subseteq \partial^2 F_r$ on $F_r$ which depends only on the projective class $[T]\in \overline{\cv}$. The dual lamination encodes, in a systematic way, the information about sequences of elements of $F_r$ with arbitrarily small translation length in $T$. We refer the reader to \cite{chl08II} for the precise technical definition of $L(T)$. For our purposes the key relevant facts are that for $T\in \overline{\widehat\cv}$ we have
$L(T)=\varnothing$ if and only if $T\in \widehat\cv$, and that whenever $T,T'\in  \overline{\widehat\cv}$ are such that $||u||_T\le ||u||_{T'}$ for every $u\in F_r$ then $L(T')\subseteq L(T)$.
Here $||u||_T$ denotes the translation length of $u\in F_r$ with respect to the action $F_r \curvearrowright T$. 
 A tree $T\in  \overline{\widehat\cv}$ is called \emph{arational} if $T\not\in\widehat\cv$ and if no leaf of $L(T)$ is carried by a proper free factor of $F_r$ \cite{ReynoldsArational}.  In this case the projectivized tree $[T]\in\partial\cv$ is also called \emph{arational}. Note that the property of being arational depends only on the dual lamination of the tree.

For $r\ge 3$, the \emph{free factor graph} $\FF$ is a simple graph where the vertex set is the set of $F_r$-conjugacy classes of proper free factors of $F_r$. Two distinct vertices of $\FF$ are adjacent in $\FF$ if and only if they can be represented as conjugacy classes $[A],[B]$ of proper free factors $A,B$ of $F_r$ such that $A\le B$ or $B\le A$. The graph $\FF$ is endowed with the simplicial metric where every edge has length 1, and with the natural left action of $\out$ by simplicial automorphisms (and hence by isometries), where for a vertex $[A]$ of $\FF$ and an element $\phi\in\out$ we have $\phi\cdot [A]=[\phi(A)]$.

It is known, by a result of Bestvina and Feighn~\cite{bf11}, that for $r\ge 2$ the free factor graph $\FF$ is Gromov-hyperbolic, and that for $\phi\in\out$ the element $\phi$ acts as a loxodromic isometry if and only if $\phi$ is fully irreducible. 
(Recall that $\phi$ is fully irreducible if no positive power of $\phi$ fixes the conjugacy class of any proper free factor.)
There is a natural coarsely defined  and coarsely $\out$-equivariant ``projection'' $\pi \colon \cv\to \FF$ where $G_0\in \cv$ is mapped to the free factor $[A]$ represented by any proper connected non-contractible subgraph of $G_0$.  It is also known~\cite{br} (see also \cite{h12}) that the hyperbolic boundary $\partial \FF$ can be identified with the set of equivalence classes $[[T]]$ of arational trees $T\in \overline{\widehat\cv}$, where two such trees $T,T'$ are considered equivalent whenever $L(T)=L(T')$.

Finally, let $\mathcal{UE}$ be the subspace of $\partial \cv$  consisting of arational trees
having a unique length measure, up to scale. More precisely, $[T] \in \mathcal{UE}$ if and only if $T$ is arational and $[T] = [T']$ whenever $L(T) = L(T')$. Such trees are sometimes called \emph{uniquely ergodic}.

\subsection{Branch-points and the geometric index of a tree}
\label{sec:index}

For an $\R$-tree $T$ and a point $p\in T$, a \emph{direction} at $p$ in $T$ is a connected component of $T\setminus\{p\}$.
The number of directions at $p$ in $T$ is denoted $\val_T(p)$ and called the \emph{valency} (or \emph{degree}) of $p$ in $T$. We think of $\val_T(p)$ as an element of $\{\infty\}\cup \{n\in \Z|n\ge 0\}$. A point $p\in T$ is a \emph{branch-point} of $T$ if $\val_T(p)\ge 3$.

Let $T\in \overline{\widehat\cv}$. In \cite{gl95} Gaboriau and Levitt proved that $T$ has only finitely many $F_r$-orbits of branch-points and only finitely many $F_r$-orbits of directions at branch-points.
They also showed that if $T\in \overline{\widehat\cv}$ is a free $F_r$-tree then for every branch-point $p\in T$ one has $\val_T(p)<\infty$. For such a free $F_r$-tree $T$, if $p_1,\dots, p_m\in T$ are representatives of all the distinct $F_r$-orbits of branch-points, \cite{gl95} defined the \emph{geometric index}  $\ind_{geom}(T)$ as
\[
\ind_{geom}(T)=\sum_{i=1}^m [\val_T(p_i)-2].
\]
The unordered list $\val_T(p_1),\dots, \val_T(p_m)$ is the \emph{index list} for $T$.

Gaboriau and Levitt further defined $\ind_{geom}(T)$ for an arbitrary (not necessarily free) tree $T \in\widehat{\overline\cv}$ and proved that one always has $\ind_{geom}(T)\le 2r-2$.
The equality $\ind_{geom}(T)= 2r-2$ holds if and only if the tree $T$ is \emph{geometric}, i.e. arises as the dual tree of a measured foliation of some finite $2$-complex with fundamental group $F_r$.
We say that $T$ is \emph{nongeometric} if $\ind_{geom}(T) < 2r-2$.
We refer the reader to the paper \cite{ch12} for more detailed background on this topic.

\subsection{Folding lines and limiting trees} \label{back:folding}
\label{back:greedyfolding}
We next turn to folding paths in $\cv$ and in $\widehat \cv$. In the case of folding paths between simplicial trees, we closely follow \cite[Section 2]{bf11}, where we refer the reader for additional details. Since we will be particularly interested in folding rays to points in $\partial \cv$, we pay special attention to this case in Proposition \ref{prop:folds_to}.

Following \cite{hm11,loneaxes}, we define a \emph{folding path} in $\widehat\cv$ as a proper continuous injective map $\gamma\colon I\to\widehat\cv$ (where $I\subseteq \R$ is an interval), with $\gamma(t)=G_t\in\widehat\cv$ for all $t\in I$, together with a family of continuous \emph{folding maps} $g_{t,t'}\colon G_t\to G_{t'}$, where $t,t'\in I$ with  $t\le t'$, satisfying the following properties: Each map $g_{t,t'}\colon G_t\to G_{t'}$ is locally injective on edges of $G_t$, and we have $g_{t,t}=Id_{G_t}$ for each $t\in I$. In addition, whenever $t\le t'\le t''$ for $t,t',t''\in I$, we have $g_{t,t''}= g_{t',t''}\circ g_{t,t'}$. We will often denote such a folding path as just $(G_t)_{t\in I}$ and suppress explicit mention of the maps $g_{t,t'}$.
A folding path is a \emph{folding line} if $I=\R$ and a \emph{folding ray} if $I=[t_0,\infty)$ for some $t_0\in \R$.

For the most part, in this paper we will concentrate on special ``greedy" types of folding paths.
We next turn to their description and refer the reader to \cite{bf11,FrancavigliaMartino} for more details.

For a point $G\in \widehat\cv$, a \emph{gate structure} $\mathcal T$ on $G$ is a partition, for every vertex $v$ of $G$, of  the set of oriented edges originating at $v$ into nonempty subsets called \emph{gates}. A turn $\{e_1,e_2\}$ at $v$ (i.e. a pair of oriented edges originating at $v$) is called \emph{legal} with respect to $\mathcal T$ if $e_1,e_2$ belong to different gates, and is called \emph{illegal} otherwise. In this setting the gate structure and the notions of legal and illegal turns naturally extend, via lifting, to $T=\widetilde G$. An edge-path (or a circuit) in $G$ is called \emph{legal} with respect to $\mathcal T$ if for every $2$-edge subpath $ee'$ of this path, the turn $\{e^{-1},e'\}$ is legal.
A \emph{train track structure} on $G$ is a gate structure $\mathcal T$ on $G$ such that at each vertex of $G$ there are at least 2 gates.

For trees $T_0\in\widehat\cv$, $T\in \overline{\widehat\cv}$, an $F_r$-equivariant map $f \colon T_0\to T$ is called a \emph{morphism} if for each edge $e=[x,y]$ of $T_0$ the map $f$ sends $e$ isometrically to $[f(x),f(y)]_T$ (so that, in particular, $f(x)\ne f(y)$). Note that a morphism is, by definition, a 1-Lipschitz map. A morphism $f \colon T_0\to T$ defines a \emph{pullback} gate structure $\mathcal T_f$ on $T_0$ where a turn $\{e_1,e_2\}$ at a vertex $x$ of $T_0$ is legal if and only if the restriction of the map $f$ to the path $e_1^{-1}e_2$ is injective. 
 A morphism $f \colon T_0\to T$ is \emph{optimal} if the pullback gate structure $\mathcal T_f$ is a train track structure on $T_0$.

Suppose $T_0=\widetilde G_0\in\widehat\cv$, $T\in \overline{\widehat\cv}$, and $f \colon T_0\to T$ is an optimal morphism. Then $f$ canonically determines in $\widehat \cv$ a \emph{greedy isometric folding path defined by $f$}, denoted  $(\widehat G_s)_{s\in J}$, with $J\subseteq [0,\infty)$ an interval starting at $0$, with $G_0=\widehat G_0$, and with the following properties and additional structure. For every $s,s'\in J$ with $s\le s'$ we have a $1$-Lipschitz map $\wh g_{s,s'}\colon \widehat G_s\to \widehat G_{s'}$ that lifts to an optimal morphism $f_{s,s'} \colon T_s\to T_{s'}$, where $T_s=\widetilde {\widehat G_s}$ and $T_{s'}=\widetilde {\widehat G_{s'}}$. For each $s\in J$ we also have an optimal morphism $f_s \colon  T_s\to T$, where $f_0=f$. These morphisms are compatible, in the sense that for every $s,s'\in J$ with $s\le s'$ we have $f_{s'}\circ f_{s,s'}=f_s$. For each $s\in J$ we equip $T_s$ with the pullback gate structure $\mathcal T_s$ induced by $f_s \colon T_s\to T$.
(In what follows, we will refer to both sets of maps $\wh g_{s,s'}$ and $f_{s,s'}$ as \emph{folding maps}.)
The ``greedy'' property of this folding line means that for each $s\in J$, which is not the right-end point of $J$, there exists an $\epsilon>0$ such that $[s,s+\epsilon)\subseteq J$ and such that for each $s'\in (s,s+\epsilon)$ the map $f_{s,s'} \colon T_s\to T_s'$ is obtained by equivariantly, at each vertex $x$ of $T_s$ and for each gate (with respect to $\mathcal T_s$) at $x$,  folding together into a single segment the initial segments of length $s'-s$ of all the edges  in that gate. The interval $J$ starting at $0$ is chosen to be maximal possible subject to $(\widehat G_s)_{s\in J}$ satisfying all these properties.

For several constructions of greedy folding lines and additional properties, see \cite[Section 2]{bf11}. We remark on a few relevant properties here. The function $\vol(T_s)$ is strictly monotone decreasing on $J$.
Moreover, the fact that $f \colon T_0\to T$ is an optimal morphism implies that for each $s\in J$ the pullback gate structure $\mathcal T_s$ on $\widehat G_s$ is a train track structure. The path $(\widehat G_s)_{s\in J}$, with the maps $\widehat g_{s,s'}$, is a folding path in $\cv$ in the more general sense described in Subsection \ref{back:folding}.
Also, in this setting, for any $s_1\le s_2$ in $J$ the path $(\widehat G_s)_{s\in [s_1,s_2]}$ is (up to shifting the parameter by $s_1$) exactly the greedy isometric folding path defined by $f_{s_1,s_2}\colon T_{s_1}\to T_{s_2}$.

It is known that if $f \colon T_0\to T$ is an optimal morphism, then the path $(\widehat G_s)_{s\in J}$ projects to a reparameterized geodesic in $\cv$ \cite{FrancavigliaMartino,a08}. In this case for $s,s'\in J$ with $s\le s'$ we have $\widehat G_s/\vol(\widehat G_s), \widehat G_{s'}/\vol(\widehat G_{s'})\in \cv$ and
\[
\dL\left(\frac{\widehat G_s}{\vol(\widehat G_s)}, \frac{\widehat G_{s'}}{\vol(\widehat G_{s'})} \right)=\log \frac{\vol(\widehat G_s)}{\vol(\widehat G_{s'})}.
\]

In particular, if $G_0\in\cv$ has volume 1, then in this setting\[
\dL \left(G_0, \frac{\widehat G_{s}}{\vol(\widehat G_{s})} \right)=\log \frac{1}{\vol(\widehat G_{s})}=-\log \vol(\widehat G_{s}).
\]
Since $\vol(\widehat G_s)$ is a strictly decreasing function on $J$, there exists a unique monotone increasing reparameterization $\alpha(t)$ of $J$ with $\alpha(0)=0$, $\alpha \colon J'\to J$, such that $\vol(\widehat G_{\alpha(t)})=e^{-t}$ for all $t\in J'$.  We denote $G_t=\widehat G_{\alpha(t)}/\vol(\widehat G_{\alpha(t)})$ for all $t\in J'$. Note that as topological spaces we have $G_t=\widehat G_{\alpha(t)}$, and the only difference between $G_t$ and $\widehat G_{\alpha(t)}$ is in their metric graph structures. For all $t\le t'$ in $J'$ we also set $g_{t,t'}=\wh g_{\alpha(t),\alpha(t')}$. Then $(G_t)_{t\in J'}$, with the maps $g_{t,t'}$, is a folding path in $\cv$ in the general sense described above.

This reparameterization gives us a path $(G_t)_{t\in J'}$ in $\cv$ starting at $G_0$ which is a geodesic in $\cv$. If $G_0\in \cv$, $T_0=\widetilde G_0$, $T\in \overline{\widehat\cv}$, and $f \colon T_0\to T$ is an optimal morphism, we refer to $(G_t)_{t\in J'}$ as the  \emph{greedy geodesic folding path defined by $f$}.

If $T\in \cv$, then in the above setting a greedy  geodesic folding path defined by $f$ always reaches $T$ in some finite time, and $J'=[0,\dL(T_0,T)]$.  If $[T]\in \partial \cv$, then it is possible that $J'$ is a finite interval (this can happen if the geodesic folding path exits $\cv$ after a finite distance), and even in the case where $J'=[0,\infty)$ we are not necessarily guaranteed that $\lim_{t\to\infty}G_t=[T]$ in $\overline\cv$. Nevertheless, for reasonably nice $T\in \partial\cv$ one can rule out such unexpected behavior.

\begin{proposition} \label{prop:folds_to}
Let $[T]\in\partial \cv$ be such that $T$ is a free $F_r$-tree. Then:
\begin{enumerate}
\item For each $r$-rose in $\cv$ there exists a metric structure $G_0\in\cv$ on this rose and an optimal morphism $f \colon \widetilde G_0=T_0\to cT$ for some $c>0$.
\item Let $T_0=\widetilde G_0\in\cv$, let $f \colon T_0\to T$ be an optimal morphism, and  let $(\widehat G_s)_{s\in J}$ and $(G_t)_{t\in J'}$ be the greedy isometric folding path and the greedy geodesic folding path determined by $f$. Denote $M=\sup \{s\mid s\in J\}$. Then:
\begin{itemize}
\item[(a)] There exists a limit $\lim_{s\to M^-} \widehat G_s=T'$ in $\overline{\widehat\cv}$, and, moreover, $T'$ is again a free $F_r$-tree and $[T']\in \partial \cv$. Moreover, in this case $L(T')\subseteq L(T)$.
\item[(b)] If, in addition, $T$ is arational, then $L(T)=L(T')$
and $J'=[0,\infty)$, so that
\[
\lim_{s\to M^-} \dL\left(G_0, \frac{\widehat G_s}{\vol(\widehat G_s)}\right)=\infty
\]
\item[(c)] If $T$ is arational and uniquely ergodic, then $T'=T$ in $\overline{\widehat\cv}$, and  hence 
\[\lim_{t\to \infty} G_t= T \] 
in $\overline\cv$.
\end{itemize}

\end{enumerate}

\end{proposition}
\begin{proof}
(1) Let $\Gamma_0\in\cv$ be an $r$-rose corresponding to a free basis $a_1,\dots, a_r$ of $F_r$. By assumption $F_r$ acts freely on $T$, so that $a_1$ is a loxodromic isometry of $T$ with translation length $||a_1||_T>0$.

Let $x_0\in\widetilde \Gamma_0$ be a lift of the vertex $v_0$ of $\Gamma_0$.  Let $L_{a_1}\subseteq T$ be the axis of $a_1$ in $T$, and pick a point $p\in L_{a_1}$.
Thus $a_1p \in L_{a_1}$ and $d_T(p,a_1p)=||a_1||_T>0$.  By replacing $T$ by $cT$ for an appropriate $c>0$ we can assume that $\sum_{i=1}^r d_T(p,a_ip)=1$.

Note that since $T$ is a free $F_r$-tree, we have $a_ip\ne p$ for $i=1,\dots, r$. We give each edge $a_i$ of $\Gamma_0$ the length $d_T(p,a_ip)>0$, which defines a new volume-1 metric structure $G_0$ on $\Gamma_0$, and a point $T_0=\widetilde G_0\in \cv$.  For $i=1,\dots, r$ denote by $x_i$ the vertex of $T_0$ which is the terminal endpoint of the lift $e_i$ of the petal $a_i$ of $\Gamma_0$ starting at $x_0$. We construct an $F_r$-equivariant morphism $f \colon T_0\to T$ by setting $f(x_0)=p$, setting $f(x_i)=a_ip$ for $i=1,\dots, r$, mapping each $e_i$ isometrically to the segment $[p,a_ip]_T$, and then extending $f$ by equivariance. By construction $f \colon T_0\to T$ is a morphism. Moreover, the fact that $p\in L_{a_1}$ implies that $x_0$ (and hence every other vertex of $T_0$) has at least 2 gates for the pullback gate structure $\mathcal T_f$. Thus $f$ is an optimal morphism, as required.

(2)

(a) Since $f \colon T_0\to T$ is an optimal morphism, hence each vertex for the pullback legal structure $\mathcal T_0$ on $T_0$ has at least 2 gates at each vertex, there exists a nontrivial $\mathcal T_0$-legal circuit $\gamma$ in $G_0$ representing the conjugacy class of some $1\ne w\in F_r$. The fact that $(\widehat G_s)_{s\in J}$ is the greedy isometric folding path determined by $f$ and starting at $G_0=\widehat G_0$ implies that for each $s\in J$ the circuit $f_{0,s}(\gamma)$ is legal in $\widehat G_s$ for the train track structure $\mathcal T_s$ induced by $f_s \colon T_s\to T$.
Recall that $M=\sup \{s\mid s\in J\}$. Thus $0<M \le \vol(G_0) < \infty$.

The fact that for any $s\le s'$ in $J$ the folding map $f_{s,s'} \colon T_s\to T_{s'}$ is 1-Lipschitz implies that for each $u\in F_r$
we have $||u||_{T_s}\ge ||u||_{T_{s'}}$.  Thus for each $u\in F_r$ the function $||u||_{T_s}$ is monotone non-increasing on $J$ and there is a finite limit $\lim_{s\to M^-} ||u||_{T_s}$. Moreover, for our legal loop $\gamma$ representing $1\ne w\in F_r$ we have $||w||_{T_s}=||w||_{T_0}>0$, and so the limit $\lim_{s\to M^-} ||w||_{T_s}=||w||_{T_0}>0$. Therefore there exists a nontrivial tree $\lim_{s\to M^-} T_s=T'$ in $\overline{\widehat\cv}$. Since there are 1-Lipschitz maps $f_s \colon T_s\to T$, we have $||u||_{T_s}\ge ||u||_T$ for every $u\in F_r$ and every $s\in J$. Therefore, for the limiting length function $||.||_{T'}$, we also have $||u||_{T'}\ge ||u||_T$ for all $u\in F_r$. Recall that $T$ is a free $F_r$-tree. Therefore for every $1\ne u\in F_r$ we have $||u||_{T'}\ge ||u||_T>0$, so that $T'$ is also a free $F_r$-tree.

We claim that $[T']\in \partial \cv$. Suppose not. Then $T'\in\widehat\cv$ and $\sup_J s=M\in J$ and $T'=T_M$.  The assumption that $T\in \partial\cv$ then implies that the map $f_M \colon  T_M\to T$ is not locally injective, and therefore for the gate structure $\mathcal T_M$ on $T_M$ there exists a gate at some vertex with at least two distinct edges in that gate. This means that the isometric folding path $(\widehat G_s)_{s\in J}$ can be continued past $s=M$ for some positive time $[M,M+\epsilon)$, contradicting the fact that $M=\sup_J s$. The condition $||u||_{T'}\ge ||u||_T$ for all $u\in F_r$ also implies that $L(T')\subseteq L(T)$.
This completes the proof of (2)(a).

(b) Suppose now that, in addition, $T$ is both free and arational.  By part (a) above we know that  $[T']\in \partial \cv$ and therefore $L(T')\ne\varnothing$. Now \cite[Proposition~4.2(i)]{br} implies that the ``derived lamination'' $L'(T)\subseteq L(T)$ is the unique minimal sublamination in $L(T)$.  Since $L'(T)$ is minimal, we have $L'(T)=L''(T)=L'''(T)$. Since $L(T')\subseteq L(T)$, and since $L(T')$ is a nonempty lamination, it follows that $L'(T)\subseteq L(T')$.  Thus $L'''(T)\subseteq L(T')$. Since $T$ is arational,  \cite[Corollary~4.3]{br} implies that $L(T')=L(T)$, and that $T'$ is also arational.

Then the greedy geodesic folding path $(G_t)_{t\in J'}$ projects to a reparameterized quasi-geodesic in the free factor complex $\mathcal{FF}$ \cite[Corollary 6.5]{bf11} which converges to a point of the hyperbolic boundary $\partial  \mathcal{FF}$ represented by $T$ \cite[Proposition 8.3]{br}. Since the projection map $\pi \colon \cv\to\mathcal {FF}$ is coarsely Lipschitz, it follows that $J'=[0,\infty)$. Indeed, otherwise $J'$ is a finite interval and $\pi$ would map the folding  line $(G_t)_{t\in J'}$ to a set of bounded diameter in $\mathcal {FF}$, which cannot limit to a point of $\partial \mathcal{FF}$. Thus indeed $J'=[0,\infty)$ and $\lim_{s\to M^-} \dL(G_0, \widehat G_s/\vol(\widehat G_s))=\infty$. Part (2)(b) is verified.

(c) Suppose now that $T$ is free arational and uniquely ergodic. By part (b) we know that $L(T)=L(T')$ and $T'$ is arational. Then, by definition of unique ergodicity, we have $[T]=[T']$ in $\partial\cv$. Thus $T'=bT$ for some $b>0$. Note that for our legal circuit $\gamma$ representing $w$ in $G_0$ we have $||w||_T=||w||_{T_0}=||w||_{T'}>0$ and therefore $b=1$. Thus $T=T'$ in $\overline{\widehat\cv}$, as required.
\end{proof}

We conclude this subsection by setting a few conventions to simplify terminology.

\begin{conv}
From now on, by a \emph{geodesic folding ray} in $\cv$ we mean a folding ray $(G_t)_{t\in [t_0,\infty)}$ in $\cv$ which, up to a shift of the parameter by $t_0$, is  a greedy geodesic folding path in $\cv$ with $J'=[0,\infty)$. Also, by a \emph{geodesic folding line} in $\cv$ we mean a folding line $(G_t)_{t\in \R}$ in $\cv$ such that for every $t_0\in R$ the path $(G_t)_{t\in [t_0,\infty)}$ is a geodesic folding ray in $\cv$.

We will often abbreviate the notation for geodesic folding rays and geodesic folding lines in $\cv$ to just $(G_t)$. Moreover, if a geodesic folding line in $\cv$ is $\varphi$-periodic for some fully irreducible $\varphi\in\out$, we usually denote such a line by $A(t)$.
\end{conv}

\subsection{Random walks and Outer space}

The general notion of a nonelementary probability measure on a group acting isometrically on a Gromov-hyperbolic metric space is discussed in more detail in Section~\ref{sec:random_rays} below. Considering the case of the action of $G=\out$ on the free factor graph $\mathcal{FF}$, a probability measure $\mu$ on $\out$ is \emph{nonelementary} if the subsemigroup $\sm$ of $\out$ generated by the support of $\mu$ contains some two independent fully irreducible elements $\psi_1,\psi_2$. Here \emph{independent} means that the attracting and repelling fixed points of $\psi_1,\psi_2$ in $\partial \mathcal{FF}$ are four distinct points.  By \cite[Proposition~2.16, Theorem~4.1]{bfh97}, fully irreducibles $\psi_1,\psi_2\in\out$ are independent if and only if $\langle \psi_1,\psi_2\rangle\le \out$ is not virtually cyclic, and also if and only if $\langle\psi_1\rangle\cap \langle\psi_2\rangle=\{1\}$.

Recall that $\mathcal{UE} \subset \partial \cv$ is the subspace of uniquely ergodic trees.

The following is Theorem 7.21 of Namazi--Pettet--Reynolds \cite{namazi2014ergodic}; see also Dahmani--Horbez \cite[Theorem 5.10]{DahmaniHorbez} and Horbez \cite[Proposition 4.4]{horbez2017central}.

\begin{theorem}[Hitting measure on $\partial \cv$]\label{th:npr}
 Let $\mu$ be a nonelementary probability measure on $\Out(F_r)$ with finite first moment with respect to $d_{\cv}$. Then for almost every sample path $\omega = (\omega_n)_{n\ge0}$ of the random walk on $(\Out(F_r), \mu)$ and any $y_0 \in \cv$, the sequence $(\omega_n  y_0)_{n\ge0}$ converges to a point $\mathrm{bnd}(\omega) \in \mathcal{UE}$. The hitting measure $\nu$ defined by setting
\[
 \nu(S) = \mathbb{P}(\mathrm{bnd}(\omega) \in S),
 \]
for all measurable subsets $S \subset \partial \cv$ is nonatomic, and it is the unique $\mu$-stationary measure on $\partial \cv$.
\end{theorem}

In fact, it is not hard to see that $\nu$--almost every $T \in \partial \cv$ is also free. Since we will need this fact, we record it here. For the statement, we recall that a fully irreducible $\varphi \in \Out(F_r)$ is \emph{geometric} if there is a once punctured surface $S$ with $\pi_1(S) = F_r$ and a pseudo-Anosov homeomorphism $f \colon S \to S$ such that $f_* = \varphi$, as outer automorphisms. If $\phi$ is not geometric, then it is \emph{nongeometric}.

\begin{corollary} \label{cor:random_free}
Suppose in addition to the hypotheses of Theorem \ref{th:npr} that the semigroup generated by the support of $\mu$ contains a nongeometric fully irreducible outer automorphism. Let $\nu$ be the associated hitting measure on $\partial \cv$ as obtained in Theorem \ref{th:npr}. Then a $\nu$-typical tree $T$ in $\partial \cv$ is free.
\end{corollary}

\begin{proof}
The hypotheses imply that $\mu$ is nonelementary with respect to the action on the co-surface graph (See \cite[Section 2.4]{TT}). By Maher--Tiozzo \cite[Theorem 1.1]{MaherTiozzo}, this means that almost every sample path converges to a point in the boundary of the co-surface graph. By work of Dowdall--Taylor \cite{DT3} the boundary of the co-surface graph is the subspace of $\partial \cv$ consisting of free and arational trees (after identifying trees with the same dual lamination, as in the identification of $\partial \mathcal{FF}$).

Now for a typical sample path $\omega$, $(\omega_n  y_0)_{n\ge0}$ converges to a point $\mathrm{bnd}(\omega) \in \mathcal{UE}$ by Theorem \ref{th:npr}. Since such a path typically projects to a path in the co-surface graph converging to a boundary point represented by a free tree, we see that $\mathrm{bnd}(\omega)$ is also free.
\end{proof}

The additional assumption in Corollary \ref{cor:random_free} on the semigroup generated by the support of $\mu$ is necessary. Without it, the entire random walk could, for example, be contained in some mapping class subgroup of $\out$ in which case almost every limiting tree has nontrivial point stabilizers.

\section{Principal outer automorphisms and \\ fellow traveling folding paths}
\label{sec:fellow_folding}

We now turn to discussing the particular type of outer automorphism, called a \emph{principal} outer automorphism,
that will act as the `seed' of our construction.  The main result of this section (Proposition \ref{prop:same_fold}) proves a strong rigidity property for folding paths that fellow travel the axis of a principal outer automorphism.

The original definition of a principal outer automorphism $\varphi \in \out$ is given in terms of its ideal Whitehead graph \cite{hm11} and the reader can find a complete definition in those terms in \cite{stablestrata} or \cite{KMPT}. Rather than recall the original definition here, we collect the essential properties that we will need and give an alternative characterization.

Recall that a fully irreducible $\varphi\in\Out(F_r)$ is called \emph{ageometric} if the attracting tree $T_+ = T_+^\varphi\in \partial \cv$ is nongeometric, i.e.  $\ind_{geom} T_+^\varphi<2r-2$.
For an ageometric fully irreducible $\varphi\in \Out(F_r)$ the action of $F_r$ on $T_+^\varphi$ is free and has dense $F_r$-orbits.
For $r\ge 3$, a fully irreducible $\varphi\in\Out(F_r)$ is \emph{principal} if $\varphi$ is ageometric with $\ind_{geom} T_+^\varphi=2r-3$, if every branch-point $p\in T_+^\varphi$ has $\val_{T_+}(p)=3$, and if every nondegenerate turn at $p$ in $T_+^\varphi$ is ``taken" by the expanding lamination $\Lambda_\varphi$ of $\varphi$.
For those readers unacquainted with this terminology, this notion essentially amounts to the fact that among all fully irreducible outer automorphisms, principal outer automorphisms are characterized as those which satisfy conditions $(2) - (4)$ in Lemma \ref{lem:principal}. We remark that principal outer automorphisms exist in $\out$ for each $r\ge 3$ \cite[Example 6.1]{stablestrata}.

As a fully irreducible outer automorphism, a principal $\varphi \in \out$ has a periodic folding line $A$ in $\cv$, which we write as $A(t)$ rather than $(A_t)$ as done in Section \ref{back:folding}. Here, $A$ is periodic in the sense that there is a $\lambda >1$ so that $ \varphi^{-1}  A(t) = A(t) \cdot \varphi = A(t + \ln \lambda)$ for all $t\in \mathbb{R}$
\footnote{Note that it is $\varphi^{-1}$ that translates along the forward `folding' direction of $A$ for the left action on $\cv$.}. Note that $\ln \lambda >0$ is the translation length of $\varphi$ in $\cv$.
We refer to $A$ as an \emph{axis} for $\varphi$.

Next we collect properties of the pair $\varphi, A$. Most of these are easily located in the literature.

\begin{lemma} \label{lem:principal}
Suppose that $\varphi \in \out$ is principal and that $A$ is an axis for $\varphi$. Then the following hold.
\begin{enumerate}

\item The folding line $A$ is the \emph{lone axis} for $\varphi$. This means that it is the unique (up to reparameterization) folding line with the property that $\lim_{t\to -\infty}A(t) = [T_-]$ and $\lim_{t\to \infty}A(t) = [T_+]$, where
 $[T_-],[T_+] \in \partial \cv$ are the repelling/ attracting trees for $\varphi$.

\item For all but a discrete collection of times, $A(t)$ is contained in the interior of a maximal simplex (i.e. it is trivalent). Moreover, when $A(t)$ is not trivalent, it has a unique vertex of degree $4$.

\item For all $t \in \mathbb{R}$, $A(t)$ has exactly one illegal turn. Hence, $A$ is a greedy folding line in the sense defined in Section \ref{back:folding}.

\item For all $t \in \mathbb{R}$ for which $A(t)$ is trivalent, every legal turn of $A(t)$ is taken (i.e. it is a turn traversed by the image of the interior of an edge of $A(s)$ under the folding map $A(s) \to A(t)$ for some $s< t$).

\end{enumerate}
\end{lemma}

\begin{proof}
Since $\vphi$ is a principal outer automorphism, by definition, its ideal Whitehead graph $IW(\vphi)$ is the disjoint union of $2r-3$ triangles. Thus, (1) is a direct consequence of \cite[Theorem 4.7]{loneaxes} and the \cite{hm11} definition of an axis bundle.

Similarly, item (2) follows immediately from Lemma 5.1 and Remark 3.11 in \cite{stablestrata}, and item (3) is explained in \cite[Remark 5.4]{KMPT} using the fact that $A$ is a lone axis for $\varphi$ (as in item (1)).

To prove item (4), recall that in the language of Section \ref{back:folding}, $A(t)$ (for $t$ greater than any fixed $t_0 \in \mathbb R$) is a greedy geodesic folding path guided by some optimal morphism $f \colon \widetilde {A(t_0) }\to T_+$, where $T_+$ is the attracting tree for $\varphi$ (as in item (1)). We suppose that $A(t_0)$ is trivalent and let $v_0$ be its unique vertex with an illegal turn (using item (3)). For any other vertex $v$ of $A(t_0)$ and any lift $\widetilde v$ to $\widetilde {A(t_0)}$, $f$ maps $\widetilde v$ to a (necessarily valence $3$) branch-point of $T_+$. From the property that $\ind_{geom} T_+ =2r-3$ we note that $f$ induces a bijection between the set of vertices of $A(t_0)$ other than $v_0$ and the set of orbits of branch-points of $T_+$.
The condition that all nondegenerate turns at $f(\widetilde v)$ are `taken' by the stable lamination means here that for each such turn there is an edge $\widetilde e$ of $\widetilde {A(t_0)}$ whose interior maps over this turn under $f$. In terms of the greedy geodesic folding line $A$, this translates to the statement that for some sufficiently large integer $n$, the folding map $A(t_0) \to A(t_0 + n \ln\lambda) = A(t_0) \cdot \varphi^n$ has the property that the image of each vertex $v \neq v_0$, which is itself a trivalent vertex with all legal turns, has each of its turns taken by some edge of $A(t_0)$.

Since $t_0$ was an arbitrary time for which $A(t_0)$ is trivalent,
using periodicity of the folding line $A$ we see that it only remains to show that the two legal turns of $v_0$ are taken by edges of $A(s)$ under the folding map $A(s) \to A(t_0)$ for some $s<t_0$.
However, this is clear by inspection: If $e_1,e_2,e_3$ are the directed edges out of $v_0$ such that $\{e_1,e_2\}$ is the unique illegal turn in $A(t_0)$, then for $i=1,2$ any open edge of $A(s)$ whose image contains $e_i$ must also contain $e_3$. Since there must be such edges of $A(s)$ for some $s<t_0$, we have that the turns $\{e_1,e_3\}$ and $\{e_2,e_3\}$ are taken, as required. This proves (4) and completes the proof of the lemma.

\end{proof}

We will also require the following lemma which states that along the axis of a principal outer automorphism, bounded length loops are legalized in bounded time. Recall that for a conjugacy class $\alpha$ in $F_r$ and graph $G \in \cv$, $\ell_G(\alpha)$ denotes the length of the immersed representative of $\alpha$ in $G$.

\begin{lemma} \label{cor:principal_legal}
Let $\varphi$ be a principal outer automorphism with lone axis $A$.
For each $l\ge 0$ there is a $D\ge 0$ such that if
$\alpha$ is a conjugacy class in $F_r$ such that $\ell_{hA(t_0)}(\alpha) \le l$ (for some $h\in \out$), then the immersed representative of $\alpha$ in $hA(t)$ is legal for all $t \ge t_0+ D$.
\end{lemma}

\begin{proof}
By applying the isometry $h \in \out$ of $\cv$, it suffices to prove the lemma for $h=1$.

There is some $t_1 \in [t_0 , t_0 + \ln \lambda]$ such that the folding map $A(t_1) \to A(t_1+\ln \lambda) = A(t_1) \cdot \varphi$, which we relabel as $f \colon \Lambda \to \Lambda$, is a train track representative of $\varphi$ mapping vertices to vertices. Note that if $\ell_{A(t_0)}(\alpha) \le l$, then $\ell_{A(t_1)}(\alpha) \le \lambda l$. 

According to \cite[Proposition 4.11]{stablestrata}, since $\varphi$ is principal there are no periodic Nielsen paths in $\Lambda$. Hence we may apply \cite[Proposition 3.1]{bf94}, which states that for any loop $\beta$ in $\Lambda$ there is an $N_\beta \ge0$ such that $[f^{N_\beta}(\beta)]$ (i.e. the tightened image of $f^{N_\beta}(\beta)$ in $\Lambda$) is legal. Let
\[
N = \max \{N_\beta \colon \ell_{A(t_1)}(\beta) \le \lambda l \}.
\]
Then our proof is completed by setting $D= (\ln \lambda) (N+1)$.
\end{proof}

We will next turn to prove our rigidity result concerning folding paths that fellow travel the lone axis $A$.
First we describe the precise definition of fellow traveling that we
will use.  

\begin{df}[Fellow traveling] Let $L \geqslant 0$ and
$\rho \geqslant 0$, and let $\gamma \colon I \to \cv$ and
$\gamma' \colon I' \to \cv$ be geodesics.
\begin{enumerate}
\item Let $t, t' \in \R$ such that $[t, t + L] \subseteq I$, and
$[t' , t' + L] \subseteq I'$, and for each $s \in [0, L]$,
$\ds (\gamma(t + s), \gamma' (t' + s)) \leqslant \rho$.  We then say
that $\gamma \vert_{[t , t + L]}$ and $\gamma \vert_{[t' ,t' + L]}$
$\rho$-fellow travel.
\item We say that $\gamma$ and $\gamma'$ $\rho$-fellow travel for
length $L$ if there exist $t$, $t'$ such that
$\gamma \vert_{[t, t + L]}$ and $\gamma \vert_{[t', t' + L]}$
$\rho$-fellow travel.

\end{enumerate}
\end{df}

We remark that here and throughout, fellow traveling in $\cv$ is
always meant with respect to the symmetric metric, and furthermore
this definition of fellowing traveling takes in to account the
orientation of the geodesic.

\smallskip

Let $(G_t)$ be a geodesic folding path.
For the statement of the next proposition, we say that a nondegenerate turn in $G_a$ is \emph{being folded} (at time $t=a)$ if the image of the turn under the folding maps $G_a \to G_b$ is degenerate for \emph{any} $b>a$.

\begin{proposition} \label{prop:same_fold}

Suppose that $\varphi \in \out$ is a principal outer automorphism with
lone axis $A$. Then there exist constants $\epsilon_0, K_0 \ge 0$ such
that if $(G_t)$, for $t \in [t_1,t_2]$, is a greedy geodesic folding
path in $\cv$, and if there is an $h \in \out$ such that $(G_t)$
$\epsilon_0$-fellow travels $A' = hA$ for length $t_2 - t_1$, then
the following holds: For any $t \in (t_1 + K_0, t_2)$ and
$s\in \mathbb{R}$ such that
\begin{itemize}
\item $G_t$ is trivalent,
\item $A'(s)$ is trivalent and in the same open simplex as $G_t$, and
\item $\phi_s \colon A'(s) \to G_t$ is a rescaling homeomorphism topologically identifying these graphs,
\end{itemize}
we have that a turn in $A'(s)$ is being folded if and only if its image under $\phi_s$ is being folded in $G_t$. Hence, $\phi_s$ preserves the train track structures in the sense that it maps legal turns to legal turns.

\end{proposition}

\begin{proof}
By applying the appropriate isometry $h \in \out$, we note that it suffices to prove the proposition for $A' = A$.

Begin by choosing $\epsilon_0 \le \log(2)$ so that $(G_t)$ passes through the same sequence of open maximal simplices as $A$. Also, fix $D \ge 0$, provided by Lemma \ref{cor:principal_legal}, to be such that any loop in $A(t)$ of length no more than $4$ is legal in $A(t+D)$.

Let $\alpha$ be a conjugacy class of $F_r$ represented by a legal loop in $G_{t_1}$ such that $\ell_{G_{t_1}}(\alpha) \le 2$. (Such an $\alpha$ is sometimes called a legal candidate in the literature.)

Since $\epsilon_0 \le \log(2)$, there is a $s_1 \in \mathbb{R}$ such that $\ds(G_{t_1, }A(s_1)) \le \epsilon_0 \le \log(2)$, and so $\ell_{A(s_1)}(\alpha) \le 4$. 
By our choice of $D$ in the above paragraph, $\alpha$ is legal in $A(s)$ for all $s \ge s_1+D$. Moreover, there is a constant $D_2 \ge D$, depending only on the axis $A$, such that $\alpha$ crosses all legal turns in $A(s)$ for all $s \ge s_1+D_2$ when $A(s)$ is trivalent. This is because when $A(s)$ is trivalent, all legal turns are taken (Lemma \ref{lem:principal}), and so the difference $D_2-D$ depends only on the stretch factor of $g$ and 
the power needed so that every edge maps over all other edges and takes all legal turns.

Hence, for all trivalent $A(s)$ with $s \ge s_1+D_2$, $\alpha$ crosses all of the legal turns in $A(s)$ and so $\alpha$ crosses all but the unique illegal turn. If $t\in [t_1,t_2]$ is such that $G_t$ lies in the same open maximal simplex as $A(s)$, then $\alpha$, which is legal in $G_t$, crosses all but one turn in $G_t$. This conclusion holds  because $\phi_s \colon A(s) \to G_t$ is a homeomorphism and so maps the immersed representative of $\alpha$ in $A(s)$ to the immersed representative of $\alpha$ in $G_t$. Hence, the one turn in $G_t$ not taken by $\alpha$ must be the unique illegal turn in $G_t$. This implies that $\phi_s \colon A(s) \to G_t$ preserves legality, whenever $s \ge s_1+D_2$ and $A(s)$ and  $G_t$ are in the same maximal open simplex.

To complete the proof of the proposition, it suffices to find a $K_0 \ge 0$ such that if $t > K_0 +t_1$, then any $A(s)$ in the same maximal open simplex with $G_t$ necessarily has $s \ge s_1+D_2$.
For this, let $0<\epsilon$ be the minimum injectivity radius (i.e. length of shortest essential loop) along the periodic line $A$. Note that if the Lipschitz distance from $G_t$ to a graph in $A$ is less than $\epsilon_0$, then the injectivity radius of $G_t$ is at least $e^{-\epsilon_0}\epsilon$. By compactness, the diameter of the subspace of a simplex consisting of graphs with injectivity radius at least $e^{-\epsilon_0}\epsilon$ is bounded by some constant $\mathfrak{D} \ge 0$. Then setting $K_0 = D_2 + \mathfrak{D} +2 \epsilon_0$ completes the proof by the triangle inequality.
\qedhere
\end{proof}

In order to apply Proposition \ref{prop:same_fold} we will require the following lemma:

\begin{lemma}\label{lem:open_balls}
Suppose that $\varphi \in \out$ is a principal outer automorphism with
lone axis $A$.
There exists $\epsilon_1 >0$ such that for every $t \in \mathbb{R}$ there is $x \in \mathbb{R}$ with $t < x < t+1$ so that the symmetric $\epsilon_1$-ball about $A(x)$ lies in the interior of a maximal simplex.
\end{lemma}

\begin{proof}
By Lemma \ref{lem:principal}.2, the set 
\[
Z = \{t \in \mathbb{R} : A(t) \text{ is not in the interior of a maximal simplex}\}
\]
 is a discrete subset of $\mathbb{R}$. Moreover, since $A(t) \cdot \varphi = A(t+\log(\lambda))$ for some $\lambda>1$, $Z$ is invariant under translation by $\log(\lambda) >0$. Hence, it suffices to assume that $t$ lies in the compact interval $[0, \log(\lambda)]$. Set $I = [0,\log(\lambda)+1]$ and note that $Z \cap I$ is finite.

Let $K$ be the complement in CV of the interiors of maximal simplices. Clearly $K$ and its closed symmetric $\epsilon$-neighborhood $N_\epsilon = N_{\le \epsilon}(K)$ are closed. 

The preimage $C_\epsilon = A^{-1}(N_\epsilon) \cap I$ is compact. It is easy to see that $\bigcap_{\epsilon>0}  C_\epsilon = Z \cap I$ 
since for any $t \notin Z$ the symmetric distance from $A(t)$ to $K$ is positive. 
Hence, we can choose $\epsilon_1>0$ sufficiently small so that each component of $C_{\epsilon_1}$ has diameter less than $1$. For such an $\epsilon_1$ and any $t \in [0, \log(\lambda)]$ there is an $x \in I$ with $t < x < t+1$  so that $x$ is not in $C_{\epsilon_1}$. Consequently, $A(x)$ has symmetric distance greater than $\epsilon_1$ from $K$ and so the symmetric $\epsilon_1$-ball about $A(x)$ is contained in the interior of a maximal simplex, as required. This completes the proof.
\end{proof}

\section{Valencies of branch-points and \\ eventually legalizing folding lines}
\label{sec:condition*}

We begin by stating a convention that we will refer to throughout this section.

\begin{conv}\label{conv:foldline}
For the remainder of this section, we assume that $[T]\in \partial\cv$ is given by a \emph{free} $F_r$-tree $T$ (where $r\ge 3$), that $G_0 \in \cv$, and that $f \colon T_0\to T$ is an optimal morphism from $T_0 = \widetilde G_0$ to $T$.
This data produces the greedy isometric folding path $(\widehat G_s)_{s\in J}$ in $\widehat\cv$ determined by $f$ starting at $\widehat G_0=G_0$.

Recall from Section \ref{back:folding} that the folding path $(\widehat G_s)_{s\in J}$ comes together with optimal morphisms $f_s\colon T_s=\widetilde{\widehat G_s}\to T$ (where $s\in J$), with ``folding maps" $\wh g_{s,s'} \colon  \wh G_s\to \wh G_{s'}$ for all $s,s'\in J, s\le s'$, and their lifts $f_{s,s'} \colon T_s\to T_{s'}$ such that $f_{s'}\circ f_{s,s'}=f_s$. We also have the corresponding geodesic folding path $(G_t)_{t\in J}$ in $\cv$.

Finally, recall that each $\widehat G_s$ is given the pullback train track structure $\mathcal T_s$ defined by the map $f_s$; although we note that because the folding path is greedy, the gate structure is unambiguous.
By part (2)(a) of Proposition~\ref{prop:folds_to}, the interval $J$ has the form $[0,M)$ for some real number $M>0$.

\end{conv}

We record the following useful general property of our folding paths.

\begin{lem}\label{lem:val}
Let $T$, $f \colon T_0\to T$, and $(\widehat G_s)_{s\in J}$  be as in Convention~\ref{conv:foldline}. Let $s\in J$ and let $x\in T_s$ be a vertex with $k\ge 3$ gates with respect to $\mathcal T_s$. Then $p=f_s(x)\in T$ is a branch-point with $\val_T(p)\ge k\ge 3$.
\end{lem}
\begin{proof}
Let $e_1,\dots, e_k$ be edges of $T_s$ originating at $x$ and representing the $k$ distinct gates at $x$. Then $f_s$ maps each $e_i$ isometrically to a nondegenerate geodesic segment $f_s(e_i)=[p,p_i]_T$ in $T$. For $i\ne j$ the edges $e_i,e_j$ are in different gates; therefore the turn $\{e_i,e_j\}$ is legal and the path $e_i^{-1}e_j$ is mapped by $f_s$ injectively to $T$. This means that for $i=1,\dots, k$ the segments $[p,p_i]_T$ represent $k$ distinct directions at $p$ in $T$. Hence $\val_T(p)\ge k\ge 3$, as required.
\end{proof}

Lemma~\ref{lem:val} motivates the following definition:

\begin{df}[Representing branch-points]
Let $T$, $f \colon T_0\to T$, and $(\widehat G_s)_{s\in J}$  be as in Convention~\ref{conv:foldline}. Let $s\in J$ and let $x\in T_s$ be a vertex with $k\ge 3$ gates with respect to $\mathcal T_s$,  and let $x_0\in V\widehat G_s$ be the projection of $x$ to $\widehat G_s$.  Let $p=f_s(x)\in T$ (so that, by Lemma~\ref{lem:val}, $p$ is a branch-point of $T$ of valency $\ge k$).

In this case we say that the branch-point $p\in T$ is \emph{represented} by $x$, and that the $F_r$-orbit of $p$ is is \emph{represented} by $x_0$.

If, moreover, $\val_T(p)= k$, we say that the branch-point $p\in T$ is \emph{faithfully represented} by $x$, and that the $F_r$-orbit of $p$ is \emph{faithfully represented} by $x_0$.
\end{df}

\begin{rk}\label{rk:rep}
Note that if a branch-point $p\in T$ is represented (resp. faithfully represented) by $x\in T_s$ then for each $s'>s$ in $J$, the branch-point $p$ is also represented (resp. faithfully represented) by $f_{s,s'}(x)\in T_{s'}$.
\end{rk}

In general it can happen that in the setting of Lemma~\ref{lem:val} the point $p=f_s(x)\in T$ has some extra directions not coming from the gates at $x$ in $T_s$, that is, that $\val_T(p)>k$, so that $p$ is represented but not faithfully represented by $x$.  (For experts: this is exactly what happens in the presence of periodic INPs in train track maps representing some nongeometric fully irreducible $\phi\in\out$.)

Below we define an additional condition satisfied by some ``good'' folding paths, which will allow us to control and ultimately rule out this kind of behavior.
This condition on folding lines is a central point of this paper.

\begin{df}[Eventually legalizing folding paths]  \label{df:*}
Let $T$, $f\colon T_0\to T$, and $(\widehat G_s)_{s\in J}$  be as in Convention~\ref{conv:foldline}. We say that the folding path  $(\widehat G_s)_{s\in J}$ is \emph{eventually legalizing} if for any $s\in J$ and any immersed finite path $\gamma$ in $\widehat G_s$, there exists $s'\in J, s'>s$ such that the tightened form $\gamma' = [g_{s,s'}(\gamma)]$ of the image of $\gamma$ in $\widehat G_{s'}$ is legal (with respect to $\mathcal T_{s'}$).  In this situation we also say that the greedy geodesic folding path $(G_t)_{t\in J'}$ in $\cv$ determined by $f$ is \emph{eventually legalizing}.
\end{df}

Note that under the assumptions of  Convention~\ref{conv:foldline}, for every $s\in J$ the subset $f_s(T_s)\subseteq T$ is an $F_r$-invariant subtree and therefore $f_s(T_s)=T$ since the action of $F_r$ on $T$ is minimal.

\begin{proposition} \label{gates_star}
Let $T$, $f \colon T_0\to T$, and $(\widehat G_s)_{s\in J}$  be as in Convention~\ref{conv:foldline}.  Assume that the greedy isometric folding path $(\widehat G_s)_{s\in J}$ is eventually legalizing.

Then for each branch-point $p\in T$ there exists some $s\in J$ and a vertex $x_0\in \widehat G_s$ such that $x_0$ faithfully represents the $F_r$-orbit of $p$.
\end{proposition}

\begin{proof}
Recall that, by the result of Gaboriau and Levitt, since $T$  is a free $F_r$-tree, every branch-point of $T$ has finite valency, and there are only finitely many $F_r$-orbits of branch-points in $T$ (see Section \ref{sec:index}).

Let $p\in T$ be a branch-point. Thus $3\le \val_T(p)=m<\infty$. Let $q_1,\dots,q_m$ be points in $T$ distinct from $p$ such that the directions at $p$ defined by geodesic segments $[p,q_1]_T,\dots,[p,q_m]_T$ represent all $m$ directions at $p$. In particular, $[p,q_i]_T \cap [p,q_j]_T = \{p\}$ for all $i \neq j$.

Recall that $T_0=\widetilde{\widehat G_0}$ and that $f=f_0:T_0\to T$ is onto. Let $u,y_1,\dots,y_m\in T_0$ be such that $f_{0}(u)=p$ and $f_{0}(y_i)=q_i$. Denote $\beta_i=[u,y_i]_{T_0}$ and denote by $\alpha_i$ the image of $\beta_i$ in $\widehat G_{0}$. Thus each $\alpha_i$ is an immersed path in $\widehat G_{0}$ from some point $v$ (the image of $u$ in $\widehat G_{0}$) to some point $z_i$ (the image of $y_i$ in $\widehat G_{0}$). Note that $f_{0}(\beta_i)$ is a path in $T$ from $p$ to $q_i$, and so this path passes over $[p,q_i]_T$ but we cannot claim yet that $f_{0}(\beta_i)=[p,q_i]_T$.

Since our folding path is eventually legalizing, there exists some $s>0$ in $J$ such that for $i=1,\dots ,m$ the tightened $\wh g_{0,s}$-image $\tau_i$ of $\alpha_i$ in $\widehat G_s$ is legal.  
All $\tau_i$ have the same initial point $v'$ which is the image of $v$ in $\widehat G_s$.

Observe that, for each $i=1,\dots, m$, the tightened $f_{0,s}$-image $\omega_i$ of $\beta_i$ in $T_s=\widetilde{\widehat G_s}$ is the lift of $\tau_i$ starting at $x = f_{0,s}(u)$ and hence legal. (Here, the map $f_{0,s}$ is as in Convention \ref{conv:foldline}.)
This means that $f_s\colon T_s\to T$ is injective on $\omega_i$.
Then $f_s(x)=p$ and $f_s(\omega_i)=[p,q_i]_T$ for $i=1,\dots ,m$.

Since we chose $q_1,\dots ,q_m$ so that the directions at the point $p$ in $T$  defined by $[p,q_1]_T,\dots,[p,q_m]_T$ are distinct, the directions defined by $\omega_1,\dots ,\omega_m$ at $x$ have to be distinct as well.
Otherwise, there would be some $i \neq j$ such that $\omega_i \cap \omega_j$ is nontrivial. But then the image of this overlap $f_s(\omega_i \cap \omega_j)$ would be nontrivial as well, implying that $[p,q_i]_T \cap [p,q_j]_T$ is nontrivial. (Recall that $f_s(\omega_i) =[p,q_i]_T$ and $f_s(\omega_j) =[p,q_j]_T$.) This contradicts our choice of distinct directions at $p$.

Since $m\ge 3$, this means that $x$ is a vertex of $T_s$, and hence $v'$ is a vertex of $\widehat G_s$, and that the directions at $v'$ represented by initial germs of $\tau_1,\dots,\tau_m$ are in $m$ distinct gates for $\mathcal T_s$.

If $v'$ has $k>m$ gates in $\widehat G_s$, that would imply that there is another direction at $x$ in $T_s$ which maps by $f_s$ to a direction at $p$ different from the $m$ directions given by $[p,q_1]_T,\dots,[p,q_m]_T$, contradicting the choice of $m$ and of $q_1,\dots,q_m$. Hence $v'$ has exactly $m$ gates in $\widehat G_s$. Thus the vertex $x\in T_s$ faithfully represents the branch-point $p\in T$, and the vertex $v'\in \widehat G_s$ faithfully represents the $F_r$-orbit of $p$, as required.
\end{proof}

We now come to the main result of this section.

\begin{theorem} \label{th:trivalent}
Let $[T]\in\partial \cv$ be a free $F_r$-tree (where $r\ge 3$), let $T_0\in\cv$, let $f \colon T_0\to T$ be an optimal morphism, and let $(\widehat G_s)_{s\in J}$ be a greedy isometric folding path in $\widehat\cv$ determined by $f$ starting at $T_0$.  Suppose that:
\begin{enumerate}
\item The folding path $(\widehat G_s)_{s\in J}$ is eventually legalizing and
\item for each $s\in J$ there exists some $s'>s$ in $J$ such that the graph $\widehat G_{s'}$ is trivalent.
\end{enumerate}
Then $T$ is trivalent and nongeometric.
\end{theorem}

\begin{proof}

Let $p\in T$ be a branch-point. Then by Proposition~\ref{gates_star} there exists some $s\in J$ and a vertex $x_0\in \widehat G_s$ such that $x_0$ faithfully represents the $F_r$-orbit of $p$. Thus $\val_T(p)=k\ge 3$, and $\widehat G_s$ has exactly $k$ gates at $x_0$ for $\mathcal T_s$. By condition (2), there exists some $s'>s$ in $J$ such that the graph $\widehat G_{s'}$ is trivalent. Then, by Remark~\ref{rk:rep}, $x_0'=\wh g_{s,s'}(x_0)\in \widehat G_{s'}$ is also a vertex with $k\ge 3$ gates that faithfully represents the $F_r$-orbit of $p$, and thus $k\le \deg_{\widehat G_{s'}}(x_0')$. Since $\widehat G_{s'}$ is trivalent, it follows that $k=3$. Thus $T$ is trivalent, as required.

We now claim that $T$ is nongeometric.
Suppose on the contrary that $T$ is geometric. Then the geometric index of $T$
is equal to $2r-2$.

Since $T$ is trivalent, and every $F_r$-orbit trivalent branch-point contributes $3-2=1$ to the geometric index of $T$,
this means that $T$ has exactly $2r-2$ $F_r$-orbits of branch-points, each of valency $3$. Let $p_1,p_2,\dots,p_{2r-2}\in T$ be representatives of these  $2r-2$ $F_r$-orbits of branch-points in $T$.

By applying Proposition \ref{gates_star}, Remark~\ref{rk:rep} and assumption (2), we can find a big enough $s\in J$ such that $\widehat G_s$ is trivalent and such that for every $i=1,\dots,2r-2$ there exists a vertex $v_i$ in $\widehat G_s$
which faithfully represents the $F_r$-orbit of $p_i$ and has exactly $3$ gates for $\mathcal T_s$. The Euler characteristic count for $\widehat G_s$ gives us  $\sum_v [(\deg(v)/2-1]=r-1$.
 We also have $\sum_{i=1}^{2r-2} [\deg(v_i)/2-1]=(2r-2)(1/2)=r-1$, which implies that $\widehat G_s$ has no other vertices and
that $V\widehat G_s=\{v_1,\dots,v_{2r-2}\}$.  Since each $v_i$ has degree $3$ and has $3$ gates in $\widehat G_s$, it follows that all non-degenerate turns at $v_i$ are legal for $i=1,\dots, 2r-2$, so that  all non-degenerate turns in $\widehat G_s$ are legal for $\mathcal T_s$. This means that $f_s \colon T_s\to T$ is locally injective, and hence an isometry, contradicting the assumption that $[T]\in \partial \cv$.
Thus $T$ is nongeometric, as claimed.
\end{proof}

The following lemma characterizes, for an eventually legalizing isometric folding line, how different vertices of $\widehat G_s$ can represent branch-points of $T$ belonging to the same $F_r$-orbit.

\begin{lem}\label{lem:br-orb}
Let $T$, $f \colon T_0\to T$, and $(\widehat G_s)_{s\in J}$  be as in Convention~\ref{conv:foldline}. Assume that the greedy isometric folding path $(\widehat G_s)_{s\in J}$ is eventually legalizing. Let $s\in J$ and let $x,y\in T_s$ be vertices with $\ge 3$ gates which are respectively lifts of vertices $x_0,y_0\in \widehat G_s$.
Let $p=f_s(x),q=f_s(y)\in T$ (so that, by Lemma~\ref{lem:val}, $p$ and $q$ are branch-points of $T$).
Then the following are equivalent:
\begin{enumerate}
\item We have $F_r p=F_rq$.
\item There exists some $s'>s$ in $J$ such that $\wh g_{s,s'}(x_0)=\wh g_{s,s'}(y_0)$.
\item There exists some $s'>s$ in $J$ and an immersed path $\gamma$ from $x_0$ to $y_0$ in $\widehat G_s$ such that the tightened image $[\wh g_{s,s'}(\gamma)]$ of $\gamma$ in $\widehat G_{s'}$ is a trivial path.
\end{enumerate}
\end{lem}

\begin{proof}
Note that (3) directly implies (2).
And (2) implies (1) as follows.
Assume that (2) holds and that $z_0=\wh g_{s,s'}(x_0)=\wh g_{s,s'}(y_0)$. Recall that we are also given a lift $f_{s',s} \colon T_s\to T_{s'}$ of $\wh g_{s,s'}$ such that $f_s=f_{s'}\circ  f_{s',s}$. Then $z_1=f_{s,s'}(x)$ and $z_2=f_{s,s'}(y)$ are both lifts of $z_0=\wh g_{s,s'}(x_0)=\wh g_{s,s'}(y_0)$.  We have $p=f_s(x)=f_{s'}\circ f_{s,s'}(x)=f_{s'}(z_1)$ and $q=f_s(y)=f_{s'}\circ f_{s,s'}(y)=f_{s'}(z_2)$. Since both $z_1,z_2$ are lifts of $z_0$, it follows that $z_2=wz_1$ for some $w\in F_r$. Since $p=f_{s'}(z_1)$ and $q=f_{s'}(wz_1)$ and since $f_{s'}$ is $F_r$-equivariant, we conclude that $q=wp$, and (1) holds.

Finally, suppose that (1) holds and $F_r p=F_rq$. Then there exists $w\in F_r$ such that $q=wp$.
Now $f_s(wx)=wf_s(x)=wp=q$. Let $\gamma$ be the projection to $\widehat G_s$ of the geodesic $[y,wx]_{T_s}$. Note that $f_s(y)=f_s(wx)=q$ in $T$. Since our folding path is eventually legalizing, there exists some $s'>s$ in $J$ such that the tightened path $\gamma'=[\wh g_{s,s'}(\gamma)]$ is legal in $\widehat G_{s'}$. If $\gamma'$ is a nontrivial path, then $\gamma'$ lifts to a legal immersed path of positive length from $f_{s,s'}(y)$ to $f_{s,s'}(wx)$ in $T_{s'}$ which maps isometrically by $f_{s'}$ to a path of positive length in $T$ from $f_{s'}(f_{s,s'}(y))$ to  $f_{s'}(f_{s,s'}(wx))$. This contradicts the fact that $f_{s'}(f_{s,s'}(y))=f_{s'}(f_{s,s'}(wx))=q$. Thus $\gamma'$ is a trivial path in $\widehat G_{s'}$.  Thus we have proved that (1) implies (3), completing the proof of the lemma.
\end{proof}

In the setting of Convention~\ref{conv:foldline}, for $s\in J$ let $V'_s\subseteq V\widehat G_s$ be the set of all vertices of $\widehat G_s$ with $\ge 3$ gates for $\mathcal T_s$. Define a relation $\sim_s$ on $V'_s$ by setting $v_1\sim_s v_2$ (for $v_1,v_2\in V_s'$) if and only if there exists $s'>s, s'\in J$ such that $\wh g_{s,s'}(v_1)=\wh g_{s,s'}(v_2)$ in $\widehat G_{s'}$. It is easy to see that $\sim_s$ is an equivalence relation on $V_s'$. Note that if $v_1\sim_s v_2$ and $v_1$ represents the $F_r$-orbit of a branch-point $p\in T$ then $v_2$ also represents the $F_r$-orbit of $p$, and $\val_T(p)\ge \max\{d_1,d_2\}$ where $d_i$ is the number of gates at $v_i$ in $\widehat G_s$ for $i=1,2$. However, in this situation if we also have that $v_1$ faithfully represents the $F_r$-orbit of $p\in T$, that does not necessarily imply that $v_2$ faithfully represents the $F_r$-orbit of $p\in T$ (since it may happen that the number of gates at $v_2$ is smaller than the number of gates at $v_1$). For a vertex $v\in V_s'$ we say that $v$ is \emph{maximal for $\sim_s$} if $v$ has the maximal number of gates among all vertices of $V_s'$ in the $\sim_s$-equivalence class of $v$.

\begin{cor}\label{cor:stab}
Let $T$, $f\colon T_0\to T$, and $(\widehat G_s)_{s\in J}$  be as in Convention~\ref{conv:foldline}. Assume that the greedy isometric folding path $(\widehat G_s)_{s\in J}$ is eventually legalizing. Let $p_1,\dots, p_m\in T$ be representatives of all the distinct $F_r$-orbits of branch-points.

There exists $s_0\in J$ such that for all $s\ge s_0$ with $s\in J$ the following holds:

\begin{enumerate}
\item There are exactly $m$ distinct $\sim_s$-equivalence classes in $V_s'$.

\item Let $v_1,\dots, v_m\in V_s'$ be representatives of all the distinct $\sim_s$-equivalence classes in $V_s'$, such that for each $i=1,\dots, m$ the vertex $v_i$ is maximal for $\sim_s$.
Then, up to re-ordering of $p_1,\dots, p_m$, for each $i=1,\dots, m$ the vertex $v_i$ faithfully represents the $F_r$-orbit of the branch-point $p_i$ of $T$. \end{enumerate}
In particular, if $k_i$ is the number of gates at $v_i$ in $\mathcal T_s$ then $k_i=\val_T(p_i)$ and
\[
\mathrm{ind}_{geom}(T)=\sum_{i=1}^m [k_i-2].
\]
\end{cor}

\begin{proof}
Proposition~\ref{gates_star} implies that there exists an $s\in J$ such that there are vertices $u_1,\dots, u_m\in V_s'$ where, for each $i$, we have that $u_i$ faithfully represents the $F_r$-orbit of $p_i$. Thus if $k_i$ is the number of gates at $u_i$ in $\mathcal T_s$ then $k_i=\val_T(p_i)\ge 3$ for $i=1,\dots, m$.
Since $p_1,\dots, p_m$ are in distinct $F_r$-orbits, Lemma~\ref{lem:br-orb} implies that for $i\ne j$ we have $u_i\not\sim_s u_j$. By Lemma~\ref{lem:val}, every vertex $v\in V_s'$ represents the $F_r$-orbit of some $p_i$, and therefore, by  Lemma~\ref{lem:br-orb}, $v\sim_s u_i$ for some $i$. Thus there are no other $\sim_s$-equivalence classes in $V_s'$ except the $m$ distinct classes given by $u_1,\dots, u_m$. This means that there are exactly $m$ distinct $\sim_s$-equivalence classes in $V_s'$, concluding the proof of (1).
Moreover, each $u_i$ is maximal in its $\sim_s$-equivalence class, since otherwise there would exist a vertex in $V_s'$ with $>k_i$ gates representing the $F_r$-orbit of $p_i$, contradicting the fact that $k_i=\val_T(p_i)$.
Thus the conclusion of part (2) in $V_s'$ holds for any maximal elements $v_1,\dots, v_m$ in the $\sim_s$-equivalence classes of $u_1,\dots, u_m$. Remark~\ref{rk:rep} and  Lemma~\ref{lem:br-orb} now imply that the conclusion of part (2) also holds for any $s'>s$ with $s'\in J$.
\end{proof}

Corollary~\ref{cor:stab} provides a precise abstract description of how an eventually legalizing folding path captures the geometric index and the index list for
the free $F_r$-tree $[T]\in\partial \cv$.

\section{Random folding rays and principal recurrence}
\label{sec:random_rays}

Fix a principal outer automorphism $\varphi \in \out$ with lone axis $A$ in $\cv$.

\begin{df}[Recurrent folding rays]
A geodesic folding ray $(G_t$) is \emph{$\varphi$-recurrent}, for some principal outer automorphism $\varphi$, if there is a $K \ge0$ such that for any $L\ge 0$, the ray $(G_t$) has a subsegment that $K$-fellow travels an $\out$-translate of $A$ for length at least $L$.

We also say that $(G_t$) is \emph{principally recurrent} if it is $\varphi$-recurrent for some principal $\varphi \in \out$.
\end{df}

The main proposition of this section is the following. It is deduced from facts about random walks on groups acting on hyperbolic space (mainly results of Maher--Tiozzo \cite{MaherTiozzo}) and the bounded geodesic image property for translates of the axis $A$, a result previously established by the authors \cite{KMPT}.

\begin{proposition} \label{prop:random_recurrent}
Suppose that $\mu$ is as in Theorem \ref{th:npr} and that $\varphi^{-1}$ is in the semigroup generated by the support of $\mu$. Let $\nu$ be the corresponding hitting measure on $\partial \cv$ (see Theorem \ref{th:npr}). Then for $\nu$ almost every tree $T\in\partial\cv$ and any geodesic folding ray $(G_t)$ converging to $T$, we have that $(G_t)$ is $\varphi$-recurrent.
\end{proposition}

We remind the reader that if $\varphi$ is principal with axis $A$ in $\cv$, then $\varphi^{-1} A(t) = A(t + \ln \lambda)$. That is, with respect to the left action on $\cv$, $\varphi^{-1}$ translates $A$ in its folding direction. \\

Before turning to the proof of Proposition \ref{prop:random_recurrent}, we briefly discuss random walks and hyperbolic spaces. The reader can find additional details in \cite{MaherTiozzo} and a similar setup in \cite{KMPT}. We assume throughout that $\mu$ is a probability measure on $G$ with finite support, although this condition is far stronger than what is needed in this section.

Now suppose we have an isometric action of a group $G$ on a
$\delta$-hyperbolic space $(X, d)$.  Recall that a
$Q$-quasigeodesic is a map $\gamma \colon I \to X$ such that
for all $s, t \in I$
\[ \tfrac{1}{Q} d( \gamma(s), \gamma(t)) - Q \leqslant \norm{t -s }
\leqslant Q d(\gamma(s), \gamma(t)) + Q. \]
We now give a definition of fellow traveling for quasigeodesics.

\begin{df}[Fellow traveling for quasigeodesics] Let $L \geqslant 0$
and $\kappa \geqslant 0$, and let $\gamma \colon I \to X$ and
$\gamma' \colon I' \to X$ be $Q$-quasigeodesics.
\begin{enumerate}
\item Let $I = [s, t]$ and $I' = [s', t']$.  We say that $\gamma$ and
$\gamma'$ $\kappa$-fellow travel if the Hausdorff distance between
$\gamma(I)$ and $\gamma'(I')$ is at most $\kappa$, and furthermore
both $d( \gamma(s), \gamma(s') ) \leqslant \kappa$ and
$d( \gamma(t), \gamma(t') ) \leqslant \kappa$.
\item We say that $\gamma$ and $\gamma'$ $\kappa$-fellow travel for
length $L$ if there exist subintervals $J \subseteq I$ and
$J' \subseteq I'$ such that $\gamma \vert_{J}$ and $\gamma \vert_{J'}$
$\kappa$-fellow travel, and furthermore the images of both $\gamma(J)$
and $\gamma'(J')$ have diameter at least $L$.
\item For a point $x = \gamma(t)$ on $\gamma$, we say that $\gamma$
and $\gamma'$ $\rho$-fellow travel for length $L$ at centered at $x$,
if there are subintervals $J \subseteq I$ and $J' \subseteq I'$, with
$t \in J$, such that $\gamma \vert_J$ and $\gamma' \vert_{J'}$
$\kappa$-fellow travel, the images of both $\gamma(J)$
and $\gamma'(J')$ have diameter at least $L$,
and, moreover, the
distance in $X$ from $x = \gamma(t)$ to each of the endpoints of
$\gamma(J)$ is at least $L/2$.
\end{enumerate}
\end{df}

We may now define what it means for two quasigeodesics to have 
an oriented match.

\begin{df}[Oriented match]
Let $\gamma \colon I \to X$ and $\gamma' \colon I' \to X$ be
quasigeodesics. We say that $\gamma$ and $\gamma'$ have an
$(L,\kappa)$--\emph{oriented match} if there is a group element 
$h\in G$ such that $\gamma$ and $h \cdot \gamma'$ $\kappa$-fellow
travel for length $L$.
\end{df}

This definition is symmetric, as if $\gamma$ and $h \cdot \gamma'$
$\kappa$-fellow travel, then $\gamma'$ and $h^{-1} \cdot \gamma'$
$\kappa$-fellow travel.

Recall that a measure $\mu$ on $G$ is nonelementary for the action $G \curvearrowright X$ if the semigroup generated by the support of $\mu$ contains $2$ loxodromic elements with distinct endpoints on $\partial X$.
Suppose that $\mu$ is a nonelementary measure for $G \curvearrowright X$
and that $\varphi \in G$ is a loxodromic in the semigroup generated by the support of $\mu$. In this setting, there is a unique $\mu$-stationary measure $\nu$ on $\partial X$, and $\nu$ is the hitting measure for the orbit of the random walk \cite[Theorem 1.1]{MaherTiozzo}. With this setup, we have the following lemma:

\begin{lemma} \label{lem:mach_hyp} %
For all $\delta \ge 0$ and all $Q \ge 1$ there is a $\kappa \ge 0$
such that the following holds: For any countable group $G$ acting on a
$\delta$-hyperbolic space $X$, with $\mu$ a nonelementary probability
measure on $G$ with finite support and hitting measure $\nu$ on
$\partial X$, then for $\nu$-almost every $\eta \in \partial X$ and
each $Q$-quasigeodesic ray $\gamma = [x_0, \eta)$ in $X$ with endpoint
$\eta$, the quasigeodesic ray $\gamma$ has, for each $L\ge 0$, an
$(L, \kappa)$--oriented match with a $Q$--quasiaxis $\alpha_\varphi$
of $\varphi$.

\end{lemma}

Here, a $Q$--\emph{quasiaxis} $\alpha_\varphi$ of $\varphi$ is a $Q$--quasigeodesic that $\varphi$ acts on by translation.

\begin{proof}
Consider the bi-infinite step space $(G, \mu)^\Z$.  Let
$S \colon (g_n)_{n \in \Z} \mapsto (g_{n+1})_{n \in \Z}$ be the shift
map, which acts ergodically on the step space.  Let
$w \colon (g_n)_{n \in \Z} \mapsto (w_n)_{n \in \Z}$ be the map from
the step space to the path space $(G^\Z, \mathbb{P})$, where
\[ w_n = \left\{ \begin{array}{ll}
                   g_1 g_2 \ldots g_n & \text{ for } n > 0 \\
                   g_0^{-1} g_1^{-1} \ldots g_{-n+1}^{-1} &
                                                            \text{ for } n \le 0,\end{array} \right.  \]
\noindent and $\mathbb{P}$ is the push forward of the product measure $\mu^\Z$
by $w$.  By \cite{MaherTiozzo}, almost every sample path converges in
both the forward and backward directions, giving rise to a map
$\partial = \partial_+ \times \partial_- \colon (G^\Z, \mathbb{P})
\to \partial X \times \partial X$, defined on a full measure subset of
the path space. In particular, this means that the shift map $S$ acts
ergodically on $(G^\Z, \mathbb{P})$, where
$S^k (w_n)_{n \in \Z} = (w_k^{-1} w_n)$.  Furthermore,
$\nu \times \check \nu$, the product of the hitting measure with the
reflected hitting measure, is the push forward of the path space
measure $\mathbb{P}$ under $\partial$.

Given an oriented $Q$-quasiaxis $\alpha_\varphi$, we shall write
$\alpha_\varphi^+$ and $\alpha_\varphi^-$ for its forward and backward
limit points in $\partial X$ respectively.  We shall write
$\alpha_\varphi(0)$ for a nearest point on $\alpha_\varphi$ to the
basepoint $x_0$ in $X$.  Given constants $\delta \ge 0$ and $Q \ge 0$,
there is a constant $\kappa \ge 0$, such that for any
$Q$-quasigeodesic $\alpha_\varphi$ in a $\delta$-hyperbolic space, and
any constant $L \ge 0$, there are open sets $A$ and $B$ in
$\partial X$, with $\alpha_\varphi^- \in A$ and
$\alpha_\varphi^+ \in B$ such that any bi-infinite $Q$-quasigeodesic
$\gamma$, with one endpoint in $A$ and the other in $B$,
$\kappa$-fellow travels length at least $L$ with the quasigeodsic
$\alpha_\varphi$, centered at $\alpha_\varphi(0)$.  Furthermore, the
distance between $\alpha_\varphi(0)$ and the closest point on $\gamma$
to the basepoint $x_0$ is bounded in terms of $\delta$ and $Q$.

We shall write $\gamma_\omega$ to denote a bi-infinite
$Q$-quasigodesic connecting the forward and backward limit points of
$(w_nx_0)_{n \in \Z}$.  If $S^k (w_n)_{n \in \Z}$ lies in
$\partial^{-1}(A \times B)$, then there is a subsegment of
$\gamma_\omega$ of length $L$, centered at the nearest point
projection of $w_k$ to $\gamma_\omega$, which fellow travels with
$w_k \alpha_\varphi$.  As $\varphi$ lies in the semigroup generated by
the support of $\mu$, by \cite[Proposition 5.4]{MaherTiozzo},
$\nu \times \check \nu (A \times B) = \nu(A) \check \nu(B)$ is
strictly positive.  In particular, $\partial^{-1}(A \times B)$ is
positive.  Therefore, by Birkhoff's pointwise ergodic theorem, the
proportion of integers $1 \le k \le N$ such that $S^k (w_n)_{n \in Z}$
lies in $\partial^{-1}(A \times B)$ converges to $\nu(A)\check \nu(B)$
as $N \to \infty$.  In particular, there is a sequence of integers
$k_i \to \infty$ such that $S^{k_i} (w_n)_{n \in \Z}$ lies in
$\partial^{-1}(A \times B)$, and as $(w_n)_{n \in \Z}$ converges to
$\partial_+(w_n)_{n \in \Z}$, this means that there are infinitely
many disjoint subintervals of $\gamma_\omega$ which
$\kappa$-fellow travel with a translate of $\alpha_\varphi$ for length $L$.  
The same property now follows for $Q$-quasigeodesic rays starting at $x_0$
and converging to
$\partial_+(w_n)_{n \in \Z}$, as every such ray has an infinite
terminal subray which fellow travels with $\gamma_\omega$.

So we have shown that for some $\kappa \ge0$ and any $L\ge0$,
the set of $\eta \in \partial X$
for which any $Q$-quasigeodesic ray $\gamma = [x_0, \eta)$ has an $(L,\kappa)$-oriented match with $\alpha_\varphi$ has $\nu$ measure $1$. Intersecting these sets over all $L \in \mathbb{Z}_+$, we see that the set of $\eta \in \partial X$ such that every $Q$-quasigeodesic ray $\gamma = [x_0, \eta)$ has an $(L, \kappa)$--oriented match with $\alpha_\varphi$ \emph{for every $L\ge 0$} also has $\nu$ measure $1$. This completes the proof.
\end{proof}

Now Proposition \ref{prop:random_recurrent} follows from Lemma \ref{lem:mach_hyp} and the bounded geodesic image property for translates of $A$. 

\begin{proof}[Proof of Proposition \ref{prop:random_recurrent}]
Recall that for $\nu$-a.e. tree $T \in  \partial \cv$, we have that $T$ is free, arational, and uniquely ergodic (Theorem \ref{th:npr} and Corollary \ref{cor:random_free}.) Hence, by Proposition \ref{prop:folds_to}, there exists a geodesic folding ray $(G_t)$ converging to $T$.

The $\pi$-image of any geodesic folding path in the free factor complex $\FF$ is a $Q$-unparameterized quasigeodesic, for $Q$ depending only on the rank of $F_r$ \cite[Corollary 6.5]{bf11}. Since $\varphi$ acts as a loxodromic isometry on $\FF$, at the expense of increasing $Q$, we may assume that the image $\pi(A)$ of the axis $A$ is a $Q$-quasiaxis for $\varphi$ in $\FF$.
So applying Lemma \ref{lem:mach_hyp} to the situation at hand, gives that almost surely the quasiray $\pi((G_t))$ has an $(L,\kappa)$-oriented match with $\pi(A)$ for every $L \ge 0$.

Unpacking this statement, we see that for any $L\ge 0$, there is an $h
\in \out$ such that $\pi((G_t))$ $\kappa$-fellow travels $\pi(hA)$ for
length at least $L$ in $\FF$.  Since the map $\pi \colon \cv \to \FF$ is coarsely Lipschitz \cite[Corollary 3.5]{bf11}, it suffices to show that fellow traveling of $\pi((G_t))$ and $\pi(hA)$ in $\FF$ can be lifted to uniform fellow traveling of $(G_t)$ and $hA$ in $\cv$. This follows from the bounded geodesic image property established in \cite[Theorem 7.8]{KMPT} and the rest of the argument is similar to the one given for \cite[Theorem A]{KMPT}.

In some detail, if $\pi((G_t))$ and $\pi(hA)$ fellow travel for length
$L$ sufficiently large, then the nearest point projection in $\FF$ of the path $\pi((G_t))$ to $\pi(hA)$ is roughly diameter $L$, depending only on $Q$ and the hyperbolicity constant of $\FF$. In terms of Outer space, this means that the projection of $(G_t)$ to the greedy folding axis $hA$ using the Bestvina--Feighn (see \cite{bf11}) projection $\mathrm{Pr}_{hA} \colon \cv \to hA$ has diameter no less than $cL$, for some $c \ge 0$ depending only on the rank of $F_r$.
This follows from the fact, established in \cite[Lemma 4.2]{dt17}, that $\pi \circ\mathrm{Pr}_{hA}$ is coarsely equal to $\bf{n} \circ \pi$, where ${\bf{n}} \colon \FF \to \pi(hA)$ is the nearest point projection.
Corollary 7.9 of \cite{KMPT} then implies that the path $(G_t)$ contains a subsegment that $K$-fellow travels a subsegment of $hA$ for length $cL -c_1$, for some constants $c_1, K \ge 0$ that depend only on the principal outer automorphism $\varphi$. Since this was true for any $L \ge 0$, we have that $(G_t)$ is $\varphi$-recurrent and the proof is complete.
\end{proof}

\section{Principally recurrent folding lines are eventually legalizing}
\label{sec:principal_recurrence}

In this section, we fix a principal outer automorphism $\varphi\in\out$ and denote by $A$ its lone folding axis in $\cv$.
Our goal is to show that principally recurrent folding paths
are all eventually legalizing.
This is achieved in Proposition \ref{recurrent_star}.

Our first lemma is proven in the same manner as Lemma 5.9 of \cite{KMPT}. It basically states that in the case of interest, if folding paths fellow travel for a long enough time, then they get arbitrarily close to one another.

\begin{lemma} \label{closer_longer}
If the greedy geodesic folding ray $(G_t)$ is $\varphi$-recurrent, then for any $\epsilon >0$ and any $L\ge 0$, the ray $(G_t)$ has a subsegment that $\epsilon$-fellow travels an $\Out(F_r)$--translate of $A$ for length at least $L$.
\end{lemma}

\begin{proof}
Using the periodicity of $A$ and $\varphi$-recurrence of $(G_t)$, we can find a $t_0 \in \mathbb{R}$ and a sequence of $h_i \in \out$ so that the rays $h_i(G_t)$ $K$-fellow travel 
the restriction of $A$ to the interval $[t_0 - L_i, t_0+L_i]$ for length $2L_i$. Here, we choose  $L_i \to \infty$ as $i \to \infty$. Up to reparameterizing the geodesic ray $h_i(G_t)$ by translation, we can assume  that $\ds(h_iG_{t_0}, A(t_0)) \le K$.

Then, just as in the proof of Lemma 5.9 of \cite{KMPT}, the sequence $h_i  (G_t)$ has a subsequence that converges uniformly on compact sets to a greedy folding line $B$ which has bounded distance from $A$ (see also \cite[Lemma 6.11]{br}). In particular, $B$ has the same limit points in $\partial \cv$ as $A$ (as in Lemma \ref{lem:principal}.1). This is to say that $B$ is a folding line from the repelling tree to the attracting tree of $\varphi$ and so since $\varphi$ is a lone axis outer automorphism we have that $B = A$, after reparameterizing. Since the convergence to $A$ along the subsequence is uniform on compact sets, we conclude that for any $\epsilon, L\ge 0$ there is an $i \ge0$ so that $h_i(G_t)$ $\epsilon$-fellow travel 
the restriction of $A$ to $[t_0 - L, t_0+L]$ for length $2L$. This completes the proof.
\qedhere
\end{proof}

The main result of this section is the following proposition.

\begin{proposition} \label{recurrent_star}
Suppose that the greedy geodesic folding ray $(G_t)$ in $\cv$ is $\varphi$-recurrent. Then $(G_t)$ is eventually legalizing.

\end{proposition}

\begin{proof}
Let $\gamma_0$ be an immersed path in $G_0$ and let $\gamma_t$ denote its image in $G_t$ (via the fold maps) after tightening. In general, if $p$ is any path in $G$, its tightening is denoted $[p]$. Our goal is to show that $\gamma_t$ is legal in $G_t$ for sufficiently large $t$.

Let $N$ be the number of illegal turns in $\gamma_0$, so that $N+1$ is the number of maximal legal segments of $\gamma_0$. Note that the number of illegal turns $N_t$ in $\gamma_t$ is nonincreasing in $t$ and so $N_t \le N$.
We begin by choosing $\mathfrak{t}_0 \ge 0$ sufficiently large so that for all $t \ge \mathfrak{t}_0$,
\begin{itemize}
\item $N_t = N_{\mathfrak{t}_0}$, i.e. the number of illegal turns has stabilized.
\end{itemize}
Hence, for all $t \ge \mathfrak{t}_0$ we have the decomposition
\begin{align} \label{decomp1}
\gamma_t = \gamma^0_t \cdot \ldots \cdot \gamma^{N_{\mathfrak{t}_0}}_t,
\end{align}
where the breakpoints happen exactly at the illegal turns of $\gamma_t$. In the language of Section 5 of \cite{bf11}, $\gamma_t$ has all \emph{surviving illegal turns} for the folding ray, in the sense that no illegal turns of $\gamma_t$ become legal or collide with one another while folding.
Although it is not strictly needed for what follows, this observation makes it clear how the decomposition of $\gamma_{t'}$ is obtained from the decomposition of $\gamma_{t}$ for $\mathfrak{t}_0\le t \le t'$: just consider the image of $\gamma^{i}_t$ under the folding map to $G_{t'}$ and remove initial and terminal portions of the image that cancel with portions of its neighbors. Since the number of illegal turns in $\gamma_t$ does not decease for $t\ge \mathfrak{t}_0$, these images are never canceled away.

Returning to the argument, by Corollary 4.8 of \cite{bf11},  for $s\ge t$ any legal segment $\sigma_t$ inside of $\gamma_t$ of length $L_t \ge 2$ gives rise to a legal segment $\sigma_{s}$ inside of $\gamma_s$ of length $L_s \ge 2 + (L_t-2) e^{s-t}$. (This conclusion follows from the so-called derivative formula of Bestvina--Feighn, \cite[Lemma~4.4]{bf11}.)
Hence, if at any time $\gamma^i_t$ 
has length more than $2$, then it grows exponentially thereafter.  So at the expense of making $\mathfrak{t}_0$ larger, we may additionally assume that for each $0 \le i \le N_{\mathfrak{t}_0}$ either:
\begin{itemize}
\item $\gamma^i_{\mathfrak{t}_0}$ has length at least $8$ (and hence has length $\ge 8$ for all $t \ge \mathfrak{t}_0$), or
\item $\gamma^i_t$ has length at most $2$ for all $t \ge \mathfrak{t}_0$.
\end{itemize}
We call the $\gamma^i_t$s of length greater than $8$ \emph{large} and the rest are called \emph{small}.

Note that if $N_{\mathfrak{t}_0} = 0$, then we are done. So assume that $N_{\mathfrak{t}_0} >0$.

Now for any $s \ge \mathfrak{t}_0$ we use (\ref{decomp1}) to construct another decomposition of $\gamma_s$,
\begin{align} \label{decomp2}
\gamma_s = r^1_s \cdot r^2_s \cdot \ldots \cdot r^k_s
\end{align}
for $k\le N_t$ defined as follows: for each large $\gamma_s^i$ there are two breakpoints of the decomposition (\ref{decomp2}) at vertices along $\gamma_s^i$ obtained by starting at the endpoints $\gamma_s^i$, moving inward (along $\gamma_s^i$) for length $2$ and choosing the next vertices of $\gamma_s^i$ (while continuing to move along $\gamma_s^i$).
Since the length of $\gamma_s^i$ is at least $8$ and every edge has length less than $1$, this process chooses two vertex breakpoints per large $\gamma_s^i$, and results in a decomposition of $\gamma_s$ in which each term begins and ends
with (possibly overlapping) legal segments of length at least $2$. We point out that $k-1$ is twice the number of large $\gamma^i_s$ in the initial decomposition of $\gamma_s$.

The decomposition of $\gamma_s$ given in $(\ref{decomp2})$ is a \emph{splitting} in the sense that if we denote the folding maps by $g_{s,t} \colon G_s \to G_t$, we have for $\mathfrak{t}_0\le s <t$
\[
\gamma_t = [g_{s,t}(r^1_s)] \cdot \ldots \cdot  [g_{s,t}(r^k_s)].
\]
This again follows from the formulation of the derivative formula stated above since legal segments of length at least $2$ are not completely cancelled under folding. (We warn the reader that we are not claiming that the above splitting of $\gamma_t$ is the same as the one appearing in (\ref{decomp2}) for $s=t$.)

Note that (for each $s\ge \mathfrak{t}_0$) the $r^j_s$'s alternate between legal segments (of length at least 2) and \emph{clusters} of segments of length no more than $3$ joined by illegal turns.
The total length of each illegal cluster is no more than
$3(N_{\mathfrak{t}_0}+1) \le 3N_0 +3$.
Moreover, if $r^j_s$ is an illegal cluster of $\gamma_s$, then for any $t>s$, $r^j_t$ is an illegal cluster of $\gamma_t$ and $r^j_t$ is a subpath of $[g_{s,t}(r^j_s)]$ whose complementary pieces are legal initial/terminal subpaths of $[g_{s,t}(r^j_s)]$. This fact follows directly from our construction.

Since $N_{\mathfrak{t}_0} >0$ and all illegal turns of $\gamma_s$ are contained in illegal clusters, there exists a $1\le j \le k$ such that $r^j_s$ is an illegal cluster for all $s \ge \mathfrak{t}_0$. We set $r_s =  r^j_s$ and henceforth work only with this illegal cluster. We will show that for some $\mathfrak{t}_0 <s<t$, the immersed path $[g_{s,t}(r_s)]$ is \emph{legal} in $G_t$. Since this is a subpath of $\gamma_t$, this shows that $N_t < N_{\mathfrak{t}_0}$; a contradiction that will complete the proof.

Now apply Lemma \ref{cor:principal_legal} with $l = 2(3N_0+8)$ to obtained a $D \ge 0$ so that for any $t \in \mathbb{R}$ and $h \in \out$,
any loop in $hA(t)$ of length at most $2(3N_0+8)$
becomes legal in $hA(t + D)$, after folding and tightening.
Also fix $\epsilon < \min\{\epsilon_0, \epsilon_1, \log(2)\}$ and $L \ge K_0 + D +2$, where $\epsilon_0$ and  $K_0$ are as in Proposition \ref{prop:same_fold} and $\epsilon_1$ is as in Lemma \ref{lem:open_balls}.
As $(G_t)$ is $\varphi$-recurrent, Lemma \ref{closer_longer} implies that for this $\epsilon , L \ge 0$, there is a interval (after time $\mathfrak{t}_0$) on which $(G_t)$ $\epsilon$-fellow travels $hA(t)$ (for some $h\in\out$) for length $L$. For ease of notation, set $A' = hA$.

Hence, we have obtained a subinterval $[\mathfrak{t}_1, \mathfrak{t}_1+L]$ ($\mathfrak{t}_1 \ge \mathfrak{t}_0$) such that the restriction of $(G_t)$ to this interval $\epsilon$-fellow travels $A'$ for length $L$.
 Applying Proposition \ref{prop:same_fold}, we get a subinterval $[\mathfrak{t}_1 + K_0, \mathfrak{t}_1+L] $ of length at least $D+2$ with the property that
for any $t \in (\mathfrak{t}_1 + K_0, \mathfrak{t_1}+L)$ and $s\in \mathbb{R}$
such that
\begin{itemize}
\item[(a)] $G_t$ is trivalent,
\item[(b)] $A'(s)$ is trivalent and in the same open simplex as $G_t$, and
\item[(c)] $\phi_s \colon A'(s) \to G_t$ is a homeomorphism topologically identifying these graphs,
\end{itemize}
we have that $\phi_s$ preserves the train track structures in the sense that it maps legal turns to legal turns.

We now choose points for which these conditions hold. 
Let $s_1,s_2 \in \mathbb{R}$ with $s_1 < s_2$ be such that the restriction of $A'$ to $[s_1,s_2]$ $\epsilon$-fellow travels the restriction of $(G_t)$ to $[\mathfrak{t}_1 + K_0, \mathfrak{t}_1+L]$. Note that each of these intervals has length at least $D+2$. 
Next apply Lemma \ref{lem:open_balls} to find $c,d \in [s_1,s_2]$ with $s_1< c<s_1+1$ and $s_2-1<d<s_2$ so that the symmetric $\epsilon_1$-balls about $A'(c)$ and $A'(d)$ are each contained in the interior of a maximal simplex. 
We record for later that $d - c \ge D$.
Finally, pick $a,b \in (\mathfrak{t}_1 + K_0, \mathfrak{t}_1+L)$ so that $\ds(G_a,A'(c))$ and $\ds(G_b,A'(d))$ are each less than $\epsilon$. As $\epsilon < \epsilon_1$, we have that $G_a$ and $A'(c)$ are contained in the same open simplex, as are $G_b$ and $A'(d)$. 

Let $\phi_{c} \colon A'(c) \to G_{a}$ and $\phi_{d} \colon A'(d) \to G_{b}$ be the homeomorphisms preserving the associated train track structures. Since $A'(c)$ has exactly one illegal turn (Lemma \ref{lem:principal}), the same is true for $G_{a}$.

Recall that the illegal cluster $r_{a}$ has length no more than $3N_0+3$ in $G_{a}$. Since there is only one illegal turn of $G_{a}$  we can easily `legally' extend $r_a$ to a immersed loop $\alpha_a$. By this we mean that $\alpha_a$ is an immersed loop containing $r_a$ so that the rest of $\alpha_a$ (call it $p_a$) is a legal arc of length at least $2$ which meets the endpoints of $r_a$ at legal turns. It is also easy to see that can be done in such a way that $\alpha_a$ has length no more than $5$ plus the length of $r_a$.

Let $\alpha$ be the conjugacy class of $\mathrm{F}_r$ represented by $\alpha_a$ in $G_a$ and let $\alpha_t$ denote the immersed representative of $\alpha$ in $G_t$ for $t\ge a$. Hence, $\ell_{G_a}(\alpha) \le 3N_0 +8$.

We claim that for all $t>a$, $[g_{a,t}(r_a)]$ is a subpath of $\alpha_t$ in $G_t$. This conclusion is an immediate consequence of the fact that $[g_{a,t}(\alpha_a)] = \alpha_t$ and the fact that
\[
\alpha_a = r_a \cdot p_a,
\]
is a splitting of $\alpha_a$ (as a loop). This last fact again follows from our construction and the formulation of the Bestvina--Feighn derivative formula used above.

We are now ready to complete the proof of Proposition~\ref{recurrent_star}. 
Using that $\ds(G_a,A'(c))\le \epsilon \le \log(2)$, we have that
\[
\ell_{A'(c)}(\alpha) \le 2 \ell_{G_a}(\alpha) \le 2(3N_0+8).
\]
Moreover, our choice of $D$ then gives that the immersed representative of $\alpha$ in $A'(c+D)$ is legal. Because $c+D \le d$, the immersed representative of $\alpha$ in $A'(d)$ is legal. But since the homeomorphism $\phi_d \colon A'(d) \to G_b$ maps the immersed representative of $\alpha$ in $A'(d)$ to the immersed representative of $\alpha$ in $G_b$ and preserves legality, the immersed representative of $\alpha$ in $G_b$ is also legal. This is all to say that $\alpha_b$ is a legal loop in $G_b$. Since $\alpha_b$ contains the path $[g_{a,b}(r_a)]$, this path too is legal in $G_b$. But this is exactly the contradiction we sought, and so the proof of Proposition~\ref{recurrent_star} is complete.
\qedhere

\section{Proof of the main result}\label{sec:main}

Recall that a probability measure $\mu$ on $\out$ is called \emph{nonelementary} if the subsemigroup $\sm$ of $\out$ generated by the support $\supp(\mu)$ of $\mu$ contains two independent fully irreducible elements (that is, two fully irreducible elements $\psi_1,\psi_2\in \out$ such that the subgroup $\langle \psi_1,\psi_2\rangle$ is not virtually cyclic).

We can now prove the main result of this paper (c.f. Theorem~\ref{th:main} in the introduction):

\begin{theorem}\label{thm:main}
Suppose that $r\ge3$ and let $\mu$ be a nonelementary probability measure on $\out$ with finite support such that $\varphi^{-1} \in \sm$ for some principal fully irreducible $\varphi\in\out$. Let $\nu$ be the hitting measure on $\partial \cv$ for the random walk $(\out,\mu)$
starting at some $y_0\in\cv$.

Then for $\nu$-a.e. $[T]\in \partial \cv$, the tree $T$ is trivalent and nongeometric.
\end{theorem}

\begin{proof}
By Corollary~\ref{cor:random_free} and Theorem \ref{th:npr}, for $\nu$-a.e. $[T]\in \partial \cv$, the tree $T$ is $F_r$-free and uniquely ergodic.

By Proposition~\ref{prop:folds_to}, there exists a (greedy) geodesic folding ray $(G_t)$ in $\cv$ such that $\lim_{t\to\infty} G_t=[T]$ in $\overline{\cv}$.
Proposition~\ref{prop:random_recurrent} now implies that the ray $(G_t)$ is $\varphi$-recurrent. Hence, by Proposition~\ref{recurrent_star},  the ray $(G_t)$ is eventually legalizing.
Therefore, by Theorem~\ref{th:trivalent}, the tree $T$ is trivalent and nongeometric.
\end{proof}

\begin{corollary}\label{cor:main}
Suppose that $r\ge3$ and let $\mu$ be a nonelementary probability measure on $\out$ with finite support such that 
$\sm$ contains a subgroup of finite index in $\out$.  Let $\nu$ be the hitting measure on $\partial \cv$ for the random walk $(\out,\mu)$ starting at some $y_0\in\cv$.

Then for $\nu$-a.e. $[T]\in \partial \cv$, the tree $T$ is trivalent and nongeometric.
\end{corollary}
\begin{proof}
Let $H\le \out$ be a subgroup of finite index such that $H\subseteq \sm$. By \cite[Example 6.1]{stablestrata}, there exists a principal fully irreducible $\varphi\in\out$. Then for some $m\ge 1$ we have $\varphi^m\in H$ and therefore $\varphi^{-m}\in \sm$. Hence, by Theorem~\ref{thm:main} above, the statement of the corollary follows.
\end{proof}

\end{proof}

\vskip1pt
\bibliographystyle{alpha}
\bibliography{References}

\end{document}